\documentclass[a4paper]{article}

\usepackage[T1]{fontenc}
\usepackage[a4paper]{meta-donnees}

\usepackage{amsmath, amssymb, amsthm}
\usepackage{geometry}
\usepackage{graphicx}
\usepackage{cancel}
\usepackage{dsfont}

\title{Equation de Dirac en espace-temps Anti-de Sitter Schwarzschild}
\author{Guillaume \textsc{Idelon-Riton}}
\date{March 2017}

\newtheorem{theo}{Théorème}[section]
\newtheorem{prop}[theo]{Proposition}
\newtheorem{lem}[theo]{Lemme}

\theoremstyle{definition}

\theoremstyle{remark}

\renewcommand\leq{\leqslant}
\renewcommand\geq{\geqslant}

\renewcommand{\Im}{\operatorname{Im}}

\newcommand\N{\mathbb N}

\newcommand\R{\mathbb R}
\newcommand\C{\mathbb C}

\newcommand{\abs}[1]{\left\lvert#1\right\rvert}
\newcommand{\norme}[1]{\left\lVert#1\right\rVert}
\newcommand{\pare}[1]{\left(#1\right)}

\DeclareMathAlphabet{\mathonebb}{U}{bbold}{m}{n}

\numberwithin{equation}{section}

\usepackage[toc,page]{appendix} 
\usepackage{hyperref}

\begin{document}

\title{Explicit formula and meromorphic extension of the resolvent for the massive Dirac operator in the Schwarzschild-Anti-de Sitter spacetime}
\author{Guillaume \textsc{Idelon-Riton}}

\maketitle

\abstract{We study the resolvent of the massive Dirac operator in the Schwarzschild-Anti-de Sitter space-time. After separation of variables, we use standard one dimensional techniques to obtain an explicit formula. We then make use of this formula to extend the resolvent meromorphically accross the real axis.}

\section{Introduction}

The study of the resolvent is an active subject of research thanks to its close link to the dynamics of the fields we wish to look at. More precisely, one way to look at the dynamics of fields is to express the propagator in terms of a contour integral in the complex plane using the resolvent. The poles of this resolvent, also called resonances or quasinormal modes, should then give the rate of decay of the fields and the frequencies at which this decay happens. The study of resonances in mathematical general relativity is then an important tool in order to understand the behaviour of the various fields that can be considered on a spacetime. This study can be traced back to the work of A. Bachelot and A. Motet-Bachelot \cite{BaMBa93} and was pursued by A. S\`{a} Barreto and M. Zworski \cite{SBZ97} for spherically symmetric black holes. Using this result, J-F. Bony and D. H\"{a}fner \cite{BoHa} were able to obtain a result concerning the local energy decay. This was then extended to more general manifolds by R. Melrose, A. S\`{a} Barreto and A. Vasy \cite{MeSBV14}. Resonances have also been employed by S. Dyatlov \cite{Dya11a}, \cite{Dya11b}, \cite{Dya12} to obtain the local energy decay of linear waves in the Kerr-de Sitter family of space-times. This was then extended to space-times close to the Kerr-de Sitter family by A. Vasy \cite{Va13}. These techniques were also adapted by A. Iantchenko to the study of resonances for Dirac fields in the Kerr-Newman-de Sitter spacetime \cite{Ia15b}. This work follows some other works by the same author concerning resonances for the Dirac fields in the de Sitter-Reissner-Nordström black holes where an expansion in terms of resonances was obtained \cite{Ia15}.
\\For Anti-de Sitter black holes, O. Gannot gave a global definition of resonances for Klein-Gordon fields in the Kerr-Anti-de Sitter family of space-times \cite{Gan14b} and localized a resonance exponentially close to the real axis \cite{Gan16} adapting a method of S. Tang and M. Zworski \cite{TaZw98}. A similar result was proven earlier by the same author for the Schwarzschild-Anti-de Sitter spacetime \cite{Gan14}. C. Warnick \cite{War14} also defined resonances for various equations in asymptotically Anti-de Sitter black hole using physical space methods. \\
Concerning the stability problem for black holes in general relativity, a precise study of the resonances was also a key point in the recent proof of the full non-linear stability of the Kerr-de Sitter family of spacetimes by P. Hintz and A. Vasy \cite{HiVa16}. The resolvent was also used by F. Finster and J. Smoller \cite{FinSmo16} to show linear stability of the Kerr family of black holes. By means of the vector field method, M. Dafermos, G. Holzegel and I. Rodnianski obtained the linear stability of the Schwarzschild solution to gravitational perturbations including precise decay estimates \cite{DaHoRo16}. G. Holzegel and J. Smulevici proved the stability of the Schwarzschild-Anti-de Sitter black hole for spherically symmetric perturbations \cite{Holsmu}. Nevertheless, the Kerr-Anti-de Sitter family seems to be unstable for reflecting boundary conditions as indicated by the logarithmic energy decay obtained by G.Holzegel and J. Smulevici in  \cite{HolSmu2} which is optimal (see \cite{HolSmu3}).\\
The Dirac equation in the Anti-de Sitter spacetime was studied by A. Bachelot in \cite{Bachelot}. Whereas the spectrum of the elliptic part is discrete in the Anti-de Sitter case, it becomes continuous when looking at the Schwarzschild-Anti-de Sitter black hole. Asymptotic completeness for the massive Dirac equation on this last spacetime was shown by the author in \cite{IR14}. Quasimodes for this same equation were constructed in \cite{IR16}.
\bigbreak
In this paper, we give an explicit formula for the resolvent of the Dirac operator in the Schwarzschild-Anti-de Sitter spacetime and show that the weighted resolvent extends meromorphically through the real axis, see section \ref{SecMainResult}. Since this space-time is spherically symmetric, after separation of variables, we are let with an operator on the half-line. We are then able to use one dimensional techniques similar to the ones employed by A. Iantchenko and E. Korotyaev for their study of the Dirac operator on the line \cite{IaKo14} and on the half-line \cite{IaKo14bis}.

\begin{center}
\bf{Acknowledgments}
\end{center}

The author acknowledges support from the ANR funding ANR-$12$-BS$01$-$012$-$01$.

\section{The Dirac equation on the Schwarzschild Anti-de Sitter spacetime}

\indent In this section, we present the Schwarzschild Anti-de Sitter space-time and give the coordinate system that we will work with in the rest of the paper. We quickly study the radial null geodesics and then formulate the Dirac equation as a system of partial differential equations.

\subsection{The Schwarzschild Anti-de Sitter space-time}

Let $\Lambda <0$. We define $l^{2}= \frac{-3}{\Lambda}$. We denote by $M$ the black hole mass.

In Boyer-Lindquist coordinates, the Schwarzschild-Anti-de Sitter metric is given by:
\begin{equation}
g_{ab}=\pare{1-\frac{2M}{r}+\frac{r^{2}}{l^{2}}} dt^{2} - \pare{1-\frac{2M}{r}+\frac{r^{2}}{l^{2}}}^{-1} dr^{2}- r^{2} \pare{d\theta^{2} + \sin^{2} \theta d\varphi^{2}}
\end{equation}
We define $F(r)= 1- \frac{2M}{r}+ \frac{r^{2}}{l^{2}}$. We can see that $F$ admits two complex conjugate roots and one real root $r=r_{SAdS}$. We deduce that the singularities of the metric are at $r=0$ and $r=r_{SAdS}=p_{+}+p_{-}$ where $p_{\pm}=\pare{Ml^{2} \pm \pare{M^{2}l^{4} + \frac{l^{6}}{27}}^{\frac{1}{2}}}^{\frac{1}{3}}$ (see \cite{Holsmu}). The exterior of the black hole will be the region $r > r_{SAdS}$ and our spacetime is then seen as $\R_{t} \times ]r_{SAdS},+\infty[ \times S^{2}$. It is well-known that the metric can be extended for $r \leq r_{SAdS}$ by a coordinate change which gives the maximally extended Schwarschild-Anti-de Sitter spacetime. In this paper, we are only interested in the exterior region.

In order to have a better understanding of this geometry, we study the outgoing (respectively ingoing) radial null geodesics (that is to say for which $\frac{dr}{dt}>0$ (respectively $\frac{dr}{dt}<0$)). Using the form of the metric we can see that along such geodesics, we have:
\begin{equation}
\frac{dt}{dr} = \pm F\pare{r}^{-1}.
\end{equation}
We thus introduce a new coordinate $x$ such that $t-x$ (respectively $t+x$) is constant along outgoing (respectively ingoing) radial null geodesics. In other words:
\begin{equation}
\frac{\mathrm dx}{\mathrm dr}= F(r)^{-1}.
\end{equation}
The coordinate system $\pare{t,x,\theta,\varphi}$ is called Regge-Wheeler coordinate system. We have:
\begin{flalign}
\hphantom{A}& \lim_{r \to r_{SAdS}} x\pare{r}=- \infty \\
\hphantom{A}& \lim_{r \to \infty} x\pare{r}= 0.
\end{flalign}
This limit proves that, along radial null geodesic, a particle goes to timelike infinity in finite Boyer-Lindquist time (recall that along these geodesic, $t-x$ and $t+x$ are constants). As a consequence, we have to put boundary conditions at $x=0$ for massless fields. For massive fields, there appears in addition a confining potential at $x=0$. For these fields there is a competition between this confining potential and the null geodesics going very fast to $x=0$. There appears a bound on the mass (related to the Breitenlohner-Freedman bound). For masses smaller than this bound, a boundary condition has to be added. For masses larger than this bound, no boundary condition is needed.

\subsection{The Dirac equation}

Using the 4-component spinor $
\psi=\begin{pmatrix}
\psi_{1}\\
\psi_{2} \\
\psi_{3} \\
\psi_{4}
\end{pmatrix}$, the Dirac equation in the Schwarzschild-Anti-de Sitter spacetime takes the form:
\begin{equation} \label{opé1}
\left( \partial_{t} + \gamma^{0}\gamma^{1} \pare{F(r)\partial_{r} + \frac{F\pare{r}}{r} + \frac{F'\pare{r}}{4}} +\frac{F(r)^{\frac{1}{2}}}{r} \cancel{D}_{\mathbb{S}^{2}} + im \gamma^{0} F(r)^{\frac{1}{2}}\right) \psi = 0. 
\end{equation}
where $m$ is the mass of the field and $\cancel{D}_{\mathbb{S}^{2}}$ is the Dirac operator on the sphere. In the coordinate system given by $\pare{\theta,\varphi}\in [0;2\pi]\times [0;\pi]$, we obtain: $\cancel{D}_{\mathbb{S}^{2}}= \gamma^{0} \gamma^{2} \pare{ \partial_{\theta} + \frac{1}{2} \cot \theta} + \gamma^{0} \gamma^{3} \frac{1}{\sin \theta} \partial_{\varphi}$. We will now work with these coordinates. For more details about how to obtain this form of the equation, we refer to a previous work \cite{IR14}.\\
Recall that the Dirac matrices $\gamma^{\mu}$, $0\leq \mu \leq 3$, unique up to unitary transform, are given by the following relations:
\begin{equation}
\gamma^{0^{*}} = \gamma^{0}; \hspace{3mm} \gamma^{j^{*}} = -\gamma^{j},\hspace{3mm} 1\leq j \leq 3;\hspace{3mm} \gamma^{\mu} \gamma^{\nu} + \gamma^{\nu} \gamma^{\mu} = 2 g^{\mu \nu}_{Mink}\mathbf{1},\hspace{3mm} 0\leq \mu, \nu \leq 3
\end{equation}
where $g^{\mu \nu}_{Mink}$ is the Minkowski metric. In our representation, the matrices take the form:
\begin{equation} \label{MatDir}
\gamma^{0} = i\begin{pmatrix}
0 & \sigma^{0} \\
-\sigma^{0} & 0
\end{pmatrix}, \hspace{2mm} 
\gamma^{k} = i \begin{pmatrix}
0 & \sigma^{k} \\
\sigma^{k} & 0
\end{pmatrix}, \hspace{2mm} k=1,2,3
\end{equation}
where the Pauli matrices are given by:
\begin{equation}
\sigma^{0}=\begin{pmatrix}
1&0 \\
0 & 1
\end{pmatrix}, \hspace{1mm} 
\sigma^{1}= \begin{pmatrix}
1 & 0 \\
0 & -1
\end{pmatrix}, \hspace{1mm}
\sigma^{2}= \begin{pmatrix}
0 & 1 \\
1 & 0
\end{pmatrix}, \hspace{1mm}
\sigma^{3}= \begin{pmatrix}
0 & -i \\
i & 0
\end{pmatrix}.
\end{equation}
We thus obtain:
\begin{equation}
\gamma^{0}\gamma^{1} = \begin{pmatrix} 
-\sigma^{1}& 0 \\
0& \sigma^{1} \\
\end{pmatrix}; \hspace{2mm} \gamma^{0}\gamma^{2} = \begin{pmatrix}
-\sigma^{2}& 0 \\
0 &\sigma^{2}
\end{pmatrix};\hspace{2mm} \gamma^{0} \gamma^{3} = \begin{pmatrix}
-\sigma^{3}& 0 \\
0& \sigma^{3}
\end{pmatrix}.\label{refGamma0,1,2,3}
\end{equation}
We make the change of spinor $\phi(t,x,\theta,\varphi) = r F(r)^{\frac{1}{4}} \psi (t,r,\theta,\varphi)$ and obtain the following equation:
\begin{equation}
\partial_{t} \phi = i \left(i\gamma^{0}\gamma^{1} \partial_{x}  + i\frac{F(r)^{\frac{1}{2}}}{r} \cancel{D}_{\mathbb{S}^{2}} -m \gamma^{0} F(r)^{\frac{1}{2}} \right) \phi.
\end{equation}
We set:
\begin{equation}
H_{m} = i\gamma^{0}\gamma^{1} \partial_{x}  + i\frac{F(r)^{\frac{1}{2}}}{r} \cancel{D}_{\mathbb{S}^{2}} -m \gamma^{0} F(r)^{\frac{1}{2}}.
\end{equation}
We introduce the Hilbert space:
\begin{equation}
\mathcal{H} := \left [L^{2}\pare{\left]-\infty,0\right[_{x} \times S^{2}_{\omega}, dx d\omega} \right ]^{4}
\end{equation}
Recall that $r$ is now a function of $x$. Using spinoidal spherical harmonics (see \cite{Bachelot} for details), we are able to diagonalize the Dirac operator on the sphere and we obtain the following operator:
\begin{equation*}
H_{m}^{s,n} = i\gamma^{0}\gamma^{1} \partial_{x}  + i\frac{F(r)^{\frac{1}{2}}}{r} \gamma^{0}\gamma^{2}\pare{s+\frac{1}{2}} -m \gamma^{0} F(r)^{\frac{1}{2}}.
\end{equation*}
In the sequel, we will write $A\pare{x} = \frac{F(r\pare{x})^{\frac{1}{2}}}{r\pare{x}}$ and $B\pare{x} = F(r\pare{x})^{\frac{1}{2}}$. The behavior of these potentials is given by:
\begin{enumerate}
\item[] \begin{flalign*}
\hphantom{A} & A\pare{x}= \begin{cases}
\frac{1}{l} + x^{2} + o\pare{x^{2}}, \hspace{2mm} x \sim 0,\\
C_{A}e^{\kappa x} + o\pare{e^{\kappa x}}, \hspace{2mm} x \sim -\infty,
\end{cases}
\end{flalign*}
 \item[] \begin{flalign*}
\hphantom{A} & B\pare{x}  = \begin{cases}
-\frac{l}{x} + x + o\pare{x}, \hspace{2mm} x \sim 0,\\
C_{B}e^{\kappa x} + o\pare{e^{\kappa x}}, \hspace{2mm} x \sim -\infty,
\end{cases}
\end{flalign*}
\end{enumerate}
where $ \kappa$ is the surface gravity, $C_{A}$ and $C_{B}$ are two positive constants. The corresponding Hilbert space is now:
\begin{equation}
\mathcal{H}_{s,n} := \left [L^{2}\pare{\left]-\infty,0\right[_{x}} \right]^{4}\otimes Y_{s,n}
\end{equation}
where $Y_{s,n}$ span the corresponding spinoidal spherical harmonic for harmonics $s$ and $n$ fixed.\\
It was proven in \cite{IR14}, that this operator is self-adjoint for all positive masses when equipped with the appropriate domain.

\section{Main result} 

\label{SecMainResult}

Let $\psi$ a solution of 
\begin{equation*}
H_{m}^{s,n} \psi = \lambda \psi
\end{equation*}
such that:
\begin{equation*}
\psi \pare{x} = \begin{pmatrix}
0 \\
e^{-i\lambda x} \\
0 \\
0
\end{pmatrix}+ \chi\pare{x}
\end{equation*}
where $\norme{\chi\pare{x}} = o\pare{
e^{\Im\pare{\lambda} x}}$ as $x$ goes to $-\infty$. We call $\psi$ a Jost solution. Let $\varphi$ a solution of the same equation satisfying the boundary conditions:
\begin{equation*}
\norme{\pare{\gamma^{1}+i}\varphi} = o\pare{\sqrt{\pare{-x}}}.
\end{equation*}
These solutions are constructed in \ref{JostSol} and \ref{BoundSol}. We introduce $\tilde{\psi} = \pare{-i} \gamma^{0} \gamma^{1} \gamma^{2} \psi$ and $\tilde{\varphi}  = \pare{-i} \gamma^{0} \gamma^{1} \gamma^{2} \varphi$ where $\gamma^{0}$, $\gamma^{1}$, $\gamma^{2}$ are the Dirac matrices \eqref{MatDir}.

\begin{theo} \label{MainTheo}
\begin{enumerate}
\item[i)]
Consider the function defined by:
\begin{flalign*}
R_{m}^{s,n}\pare{x,y,\lambda} & = \pare{\varphi\pare{x} \psi^{t}\pare{y} + \tilde{\varphi}\pare{x} \tilde{\psi}^{t}\pare{y}} M_{\alpha,\beta}^{-1} i\Gamma^{1} \mathds{1}_{]-\infty,x[}\pare{y} \\
& \quad + \pare{\psi\pare{x}\varphi^{t}\pare{y} + \tilde{\psi}\pare{x} \tilde{\varphi}^{t}\pare{y}}M_{\alpha,\beta}^{-1} i\Gamma^{1}\mathds{1}_{]x,0[}\pare{y}.
\end{flalign*}
where $\alpha = \varphi_{1} \psi_{2} - \psi_{1} \varphi_{2} + \varphi_{3} \psi_{4} - \psi_{3} \varphi_{4}$, $\beta = \varphi_{1} \psi_{3} - \psi_{1} \varphi_{3} + \varphi_{2} \psi_{4} - \psi_{2} \varphi_{4}$ and:
\begin{equation*}
M_{\alpha,\beta} =\begin{pmatrix}
0 & \alpha & \beta & 0 \\
-\alpha & 0 & 0 & \beta \\
-\beta & 0 & 0 & \alpha \\
0 &  -\beta & -\alpha & 0 
\end{pmatrix}. 
\end{equation*}
Here, $\varphi_{i}$ and $\psi_{i}$ are the components of $\varphi$ and $\psi$. Let:
\begin{equation*}
R_{m}^{s,n}\pare{\lambda}f \pare{x} = \int_{-\infty}^{0} R_{m}^{s,n}\pare{x,y,\lambda}f\pare{y} dy.
\end{equation*}
Then, for all $\lambda \in \C$ such that $\Im\pare{\lambda} >0$, we have:
\begin{equation*}
\pare{H_{m}^{s,n} - \lambda}^{-1} = R_{m}^{s,n}\pare{\lambda}.
\end{equation*}
\item[ii)] Now, let $f_{\epsilon}\pare{x} = e^{\epsilon x}$. Then the operator $f_{\epsilon} \pare{H_{m}^{s,n} - \lambda}^{-1} f_{\epsilon}$ defined for $\verb?Im? \pare{\lambda} >0$ extend meromorphically to $\{\lambda \in \C \hspace{1mm} \vert \hspace{1mm} \verb?Im? \pare{\lambda} > - \epsilon \}$ for all $0< \epsilon < \frac{\kappa}{2}$ where $\kappa$ is the surface gravity. The poles of this meromorphic extension are called resonances.
\end{enumerate}
\end{theo}

\section{Jost solutions} 

\label{JostSol}

In this section, we are interested in the construction of the Jost solution presented in the last section. We have:
\begin{prop}\label{propJost}
For all $\lambda \in \C$ such that $\Im\pare{\lambda}>-\frac{\kappa}{2}$, there exist solutions $\varphi_{2},\varphi_{3}$ to the equation:
\begin{equation*}
H_{m}^{s,n} \varphi = \lambda \varphi
\end{equation*}
such that:
\begin{equation*}
\varphi_{2} \pare{x} = \begin{pmatrix}
0 \\
e^{-i\lambda x} \\
0 \\
0
\end{pmatrix}+ \phi_{2}\pare{x};
\hspace{5mm} \varphi_{3} = \begin{pmatrix}
0 \\
0 \\
e^{-i \lambda x} \\
0
\end{pmatrix} + \phi_{3}\pare{x} 
\end{equation*}
with $\norme{\phi_{2}\pare{x}} = o\pare{e^{\Im\pare{\lambda}x}}$ and $\norme{\phi_{3}\pare{x}} = o\pare{e^{\Im\pare{\lambda}x}}$ as $x$ goes to $-\infty$. Moreover, we have:
\begin{flalign*}
\norme{\varphi_{2}\pare{x}} & \leq e^{\Im\pare{\lambda}x} e^{\int_{-\infty}^{x} e^{\max\pare{0,2\Im\pare{\lambda}t}} \norme{V_{m}\pare{t}} dt}, \\
\norme{\varphi_{3}\pare{x}} &\leq e^{\Im\pare{\lambda}x} e^{\int_{-\infty}^{x}e^{\max\pare{0,2\Im\pare{\lambda}t}} \norme{V_{m}\pare{t}} dt}
\end{flalign*}
for $\Im\pare{\lambda} > - \frac{\kappa}{2}$
\end{prop}

\begin{proof}
We prove the proposition for $\varphi_{2}$, the case of $\varphi_{3}$ can be treated in the same way. We write:
\begin{equation*}
V_{m} \pare{x} = H_{m}^{s,n} - H_{c}=\pare{s+\frac{1}{2}}\gamma^{0}\gamma^{2} A\pare{x} - m \gamma^{0} B\pare{x}.
\end{equation*}
The equation can be put under the form:
\begin{equation*}
\partial_{x} \varphi = i \lambda \Gamma^{1} \varphi \pare{x} - i \Gamma^{1} V_{m}\pare{x} \varphi\pare{x}.
\end{equation*}
We introduce the fundamental matrix of solutions:
\begin{equation*}
\mathcal{M}_{c}\pare{x} =\text{diag}\pare{e^{i\lambda x}, e^{-i\lambda x}, e^{-i \lambda x}, e^{i \lambda x}},
\end{equation*}
which satisfies:
\begin{equation*}
\partial_{x} \pare{\mathcal{M}_{c}}\pare{x} = i\lambda \Gamma^{1} \mathcal{M}_{c}\pare{x},
\end{equation*}
and the relations
\begin{equation*}
\mathcal{M}_{c}\pare{x}\mathcal{M}_{c}\pare{t}  = \mathcal{M}_{c}\pare{x+t}, \hspace{5mm}
\mathcal{M}_{c}\pare{-x}  = \mathcal{M}_{c}\pare{x}^{-1}, \hspace{5mm}
\Gamma^{1} \mathcal{M}_{c}\pare{x}  = \mathcal{M}_{c}\pare{x} \Gamma^{1}.
\end{equation*}
Using these relations, we obtain:
\begin{equation*}
\partial_{x} \pare{\mathcal{M}_{c} \pare{-x} \varphi\pare{x} }= -i \mathcal{M}_{c}\pare{-x} \Gamma^{1} V_{m}\pare{x} \varphi\pare{x}.
\end{equation*}
We consider the associated integral equation:
\begin{equation*}
\mathcal{M}_{c}\pare{-x} \varphi_{2}\pare{x} = \begin{pmatrix}
0 \\
1 \\
0 \\
0
\end{pmatrix}
- i \int_{-\infty}^{x} \mathcal{M}_{c} \pare{-t} \Gamma^{1} V_{m}\pare{t} \varphi_{2} \pare{t} dt
\end{equation*}
which gives:
\begin{equation*}
\varphi_{2}\pare{x} = \begin{pmatrix}
0 \\
e^{-i\lambda x} \\
0 \\
0
\end{pmatrix}
- i \int_{-\infty}^{x} \mathcal{M}_{c} \pare{x-t} \Gamma^{1} V_{m}\pare{t} \varphi_{2}\pare{t} dt.
\end{equation*}
We look for a solution expressed as a series:
\begin{equation*}
\varphi_{2}\pare{x} = \underset{n\geq 0}{\sum} \varphi_{2,n}\pare{x}.
\end{equation*}
We obtain the following equations:
\begin{flalign*}
\varphi_{2,0} \pare{x} & = \begin{pmatrix}
0 \\
e^{-i\lambda x} \\
0 \\
0 
\end{pmatrix} \\
\varphi_{2,n+1}\pare{x} & = -i \int_{-\infty}^{x} \mathcal{M}_{c}\pare{x-t} \Gamma^{1} V_{m}\pare{t} \varphi_{2,n}\pare{t} dt.
\end{flalign*}
Since $V_{m}$ is behaving like $e^{\kappa x}$ at $-\infty$ and $\Im\pare{\lambda} > - \frac{\kappa}{2}$, these integrals are well defined. Moreover, we have:
\begin{equation*}
\norme{\mathcal{M}_{c}\pare{x-t}} \leq e^{\abs{\Im\pare{\lambda}} \pare{x-t}}
\end{equation*}
where the norm is the supremum of the modulus of the coefficients and $t \leq x$. We now investigate the two cases $\Im\pare{\lambda} \geq 0$ and $0>\Im\pare{\lambda}>-\frac{\kappa}{2}$.
\begin{enumerate}
\item[1)] We suppose here that $\Im\pare{\lambda} \geq 0$. We will show, by induction, that $\norme{\varphi_{2,n}\pare{x}} \leq e^{\Im\pare{\lambda}x}\frac{1}{n!}\pare{\int_{-\infty}^{x}\norme{V_{m}\pare{t}}dt}^{n}$ for all $n\in \N$.\\
Indeed, this is true for $\varphi_{2,0}$.\\
We then suppose that $\norme{\varphi_{2,n}\pare{x}} \leq e^{\Im\pare{\lambda}x} \frac{1}{n!}\pare{\int_{-\infty}^{x}\norme{V_{m}\pare{t}}dt}^{n}$ for some $n \in \N$. Then:
\begin{flalign*}
\norme{\varphi_{2,n+1}\pare{x}} 
& \leq \int_{-\infty}^{x} e^{\Im\pare{\lambda}\pare{x-t}}\norme{V_{m}\pare{t}} e^{\Im\pare{\lambda}t}\frac{1}{n!} \pare{\int_{-\infty}^{t} \norme{V_{m}\pare{t'}} dt'}^{n} dt \\
& = e^{\Im\pare{\lambda}x}\int_{-\infty}^{x} \frac{1}{\pare{n+1}!} \frac{\partial}{\partial t}\pare{\int_{-\infty}^{t} \norme{V_{m}\pare{t'}}dt'}^{n+1} dt \\
& = e^{\Im\pare{\lambda}x}\frac{1}{\pare{n+1}!} \pare{\int_{-\infty}^{x} \norme{V_{m}\pare{t}} dt}^{n+1}.
\end{flalign*}
The serie $\varphi_{2}\pare{x} = \underset{n\geq 0}{\sum} \varphi_{2,n}\pare{x}$ is then converging uniformly on every compact set and gives a solution to our equation. Furthermore, we obtain the estimate:
\begin{equation*}
\norme{\varphi_{2}\pare{x}} \leq e^{\Im\pare{\lambda}x} e^{\int_{-\infty}^{x} \norme{V_{m}\pare{t}} dt}.
\end{equation*}
\item[2)] We now suppose that $0> \Im\pare{\lambda} > - \frac{\kappa}{2}$. In this case, we show that $\norme{\varphi_{2,n}\pare{x}} \leq e^{-\Im\pare{\lambda}x}\frac{1}{n!}\pare{\int_{-\infty}^{x}e^{2\Im\pare{\lambda}t}\norme{V_{m}\pare{t}}dt}^{n}$ for all $n >0$. Indeed, we have:
\begin{equation*}
\varphi_{2,1}\pare{x} =  -i \int_{-\infty}^{x} \mathcal{M}_{c}\pare{x-t} \Gamma^{1} V_{m}\pare{t} \varphi_{2,0}\pare{t} dt,
\end{equation*}
and $\norme{\mathcal{M}_{c}\pare{x-t}} \leq e^{\Im\pare{\lambda}\pare{t-x}}$ so that:
\begin{equation*}
\norme{\varphi_{2,1}} \leq e^{-\Im\pare{\lambda}x} \int_{-\infty}^{x} e^{2\Im\pare{\lambda}t} \norme{V_{m}\pare{t}} dt.
\end{equation*}
This last integral is well-defined using that $V_{m}$ is decaying like $e^{\kappa x}$ at $-\infty$. Supposing that we have $\norme{\varphi_{2,n}\pare{x}} \leq e^{-\Im\pare{\lambda}x}\frac{1}{n!}\pare{\int_{-\infty}^{x}e^{2\Im\pare{\lambda}t}\norme{V_{m}\pare{t}}dt}^{n}$ for some $n>0$, a similar argument as for $\Im\pare{\lambda} \geq 0$ gives that:
\begin{equation*}
\norme{\varphi_{2,n+1}\pare{x}} \leq e^{-\Im\pare{\lambda}x} \frac{1}{\pare{n+1}!} \pare{\int_{-\infty}^{x} e^{2\Im\pare{\lambda}t}\norme{V_{m}\pare{t}} dt}^{n+1}.
\end{equation*}
This proves the convergence of the serie and the estimate we wanted.
\end{enumerate}
In any case, since every $\varphi_{2,n}$ is analytic (because of the exponential term in $\lambda$) and the sum is uniformly convergent on every compact set, we conclude that the solution obtained is also analytic for $\Im\pare{\lambda} > - \frac{\kappa}{2}$.
\end{proof}

\section{Solutions satisfying boundary conditions}

\label{BoundSol}

In this section, we are interested in finding solutions to the equation:
\begin{equation*}
H_{m}^{s,n} \varphi = \lambda \varphi
\end{equation*}
satisfying the boundary conditions:
\begin{equation*}
\norme{\pare{\gamma^{1}+i}\varphi} = o\pare{\sqrt{\pare{-x}}}.
\end{equation*}
Recall, from \cite{IR14}, that, for $2ml<1$, $H_{m}^{s,n}$ is self-adjoint with domain:
\begin{equation*}
D\pare{H_{m}^{s,n}} = \{\varphi \in \mathcal{H} \lvert H_{m}^{s,n}\varphi \in \mathcal{H}, \hspace{2mm} \norme{\pare{\gamma^{1}+i}\varphi} = o\pare{\sqrt{\pare{-x}}}\},
\end{equation*}
where $\mathcal{H} = \left [L^{2}\pare{]-\infty,0[} \right ]^{4}$. These boundary conditions can be rewritten as:
\begin{equation*}
\abs{\pare{\varphi_{1} + \varphi_{3}}\pare{x}} = o\pare{\sqrt{\pare{-x}}}, \hspace{5mm}
 \abs{\pare{\varphi_{2} - \varphi_{4}}\pare{x}} = o\pare{\sqrt{\pare{-x}}}.
\end{equation*}
We will prove the:
\begin{prop}\label{BCS1}
We suppose that $2ml<1$ and that $\lambda \in \C$ satisfies $\Im\pare{\lambda}>- \frac{\kappa}{2}$. We can find a solution $\varphi$, analytic for $\Im\pare{\lambda}>- \frac{\kappa}{2}$, to the equation:
\begin{equation*}
H_{m}^{s,n}\varphi =  \lambda \varphi
\end{equation*}
such that $\norme{\pare{\gamma^{1}+i}\varphi\pare{x}}= o\pare{\sqrt{\pare{-x}}}$. Moreover, we have the following estimate:
\begin{equation*}
\norme{\varphi\pare{x}}\leq 4 N \pare{-x}^{-ml} e^{-\frac{6 C_{\lambda,m} \pare{s+\frac{1}{2}}}{1-2ml} x}
\end{equation*}
for all $x\in ]-\infty,0[$, with $N = 2 \max \pare{{\abs{c},\abs{d}}}$, where $c = \underset{x\to 0}{\lim} \frac{1}{2} \pare{-x}^{ml}\pare{\varphi_{1}-\varphi_{3}}$,\\
$d = \underset{x\to 0}{\lim} \frac{1}{2} \pare{-x}^{ml}\pare{\varphi_{2}+\varphi_{4}}$, and $C_{\lambda,m} = \underset{x\in ]-\infty,0[}{\max} \pare{\abs{\lambda},\abs{m\pare{B\pare{x}+\frac{l}{x}}},\abs{A\pare{x}}}$. When $x \to -\infty$, we have the estimate:
\begin{equation*}
\norme{\varphi \pare{x}} \leq C e^{-\abs{\Im\pare{\lambda}}x}.
\end{equation*}
\end{prop}

\begin{proof}
First, we write the equation under the form:
\begin{flalign*}
H_{m}^{s,n}\varphi = \lambda \varphi 
& \Leftrightarrow \partial_{x} \varphi +i \pare{s+\frac{1}{2}} \gamma^{1}\gamma^{2} A\pare{x} \varphi - im \gamma^{1} B\pare{x} \varphi = i \lambda \Gamma^{1} \varphi \\
& \Leftrightarrow \partial_{x} \varphi +i\frac{ml}{x} \gamma^{1} \varphi = i\lambda \Gamma^{1} \varphi - i \pare{s + \frac{1}{2}} \gamma^{1}\gamma^{2} A\pare{x} \varphi + im \gamma^{1} \pare{B\pare{x} + \frac{l}{x}} \varphi.
\end{flalign*}
We write:
\begin{equation*}
V_{\lambda,m}\pare{x} = i\lambda \Gamma^{1} - i \pare{s + \frac{1}{2}} \gamma^{1}\gamma^{2} A\pare{x} + im \gamma^{1} \pare{B\pare{x} + \frac{l}{x}}.
\end{equation*}
Since
\begin{flalign*}
A\pare{x} & =\frac{1}{l} + x^{2} + o\pare{x^{2}}, \hspace{2mm} x \sim 0,\\
B\pare{x} &=
-\frac{l}{x} + x + o\pare{x}, \hspace{2mm} x \sim 0,
\end{flalign*}
$V_{\lambda,m}$ is bounded near $0$. We are now studying the equation:
\begin{equation} \label{EqVpBor}
\partial_{x} \varphi +i \frac{ml}{x} \gamma^{1} \varphi = V_{\lambda,m}\pare{x} \varphi,
\end{equation}
where
\begin{equation*}
\gamma^{1} = i\begin{pmatrix}
0 & 0 & 1 & 0\\
0 & 0 & 0 & -1 \\
1 & 0 & 0 & 0 \\
0 & -1 & 0 & 0
\end{pmatrix}.
\end{equation*}
We first study the equation:
\begin{equation*}
\partial_{x} \psi + i \frac{ml}{x} \gamma^{1} \psi = 0.
\end{equation*}
We introduce the matrices:
\begin{equation*}
P = \frac{1}{\sqrt{2}}\begin{pmatrix}
1 & 0 & 1 & 0 \\
0 & 1 & 0 & 1 \\
1 & 0 & -1 & 0 \\
0 & -1 & 0 & 1
\end{pmatrix}; \hspace{5mm}
P^{-1} = \frac{1}{\sqrt{2}} \begin{pmatrix}
1 & 0 & 1 & 0 \\
0 & 1 & 0 & -1 \\
1 & 0 & -1 & 0 \\
0 & 1 & 0 & 1
\end{pmatrix},
\end{equation*}
such that:
\begin{equation*}
\gamma^{1} = P \begin{pmatrix}
i I_{2} & 0 \\
0 & -i I_{2}
\end{pmatrix} P^{-1}.
\end{equation*}
We obtain:
\begin{equation*}
\psi \pare{x} = P \begin{pmatrix}
\pare{-x}^{ml} & 0 & 0 & 0 \\
0 & \pare{-x}^{ml} & 0 & 0 \\
0 & 0 & \pare{-x}^{-ml} & 0 \\
0 & 0 & 0 & \pare{-x}^{-ml} \\
\end{pmatrix} P^{-1} \psi_{-1}
\end{equation*}
where $\psi_{-1}$ is a condition given on the value of $\psi$ at $-1$. We write:
\begin{equation*}
\mathcal{M}_{0}\pare{x} = P \begin{pmatrix}
\pare{-x}^{ml} & 0 & 0 & 0 \\
0 & \pare{-x}^{ml} & 0 & 0 \\
0 & 0 & \pare{-x}^{-ml} & 0 \\
0 & 0 & 0 & \pare{-x}^{-ml} \\
\end{pmatrix} P^{-1},
\end{equation*}
which satisfies the conditions
\begin{equation*}
\mathcal{M}_{0} \pare{x}^{-1} = \mathcal{M}_{0}\pare{\frac{1}{x}}; \hspace{5mm} \mathcal{M}_{0}\pare{-1} = I_{4}; \hspace{5mm} \mathcal{M}_{0}\pare{x} \gamma^{1} = \gamma^{1} \mathcal{M}_{0}\pare{x}.
\end{equation*}
Moreover, we have:
\begin{equation*}
\partial_{x}\pare{\mathcal{M}_{0}\pare{\frac{1}{x}}} = i \frac{ml}{x} \gamma^{1} \mathcal{M}_{0}\pare{\frac{1}{x}}.
\end{equation*}
Thus, if $\varphi$ is a solution of \eqref{EqVpBor}, then:
\begin{equation*}
\partial_{x} \pare{\mathcal{M}_{0}\pare{\frac{1}{x}} \varphi\pare{x}}  = \mathcal{M}_{0}\pare{\frac{1}{x}} V_{\lambda,m}\pare{x} \varphi\pare{x}.
\end{equation*}
Some elementary calculations give:
\begin{equation*}
\mathcal{M}_{0}\pare{\frac{1}{x}}\varphi = \frac{1}{2}\begin{pmatrix}
\pare{-x}^{-ml}\pare{\varphi_{1} + \varphi_{3}} + \pare{-x}^{ml}\pare{\varphi_{1}-\varphi_{3}} \\
\pare{-x}^{-ml} \pare{\varphi_{2} - \varphi_{4}} + \pare{-x}^{ml}\pare{\varphi_{2}+\varphi_{4}} \\
\pare{-x}^{-ml}\pare{\varphi_{1}+\varphi_{3}} + \pare{-x}^{ml}\pare{\varphi_{3}-\varphi_{1}} \\
\pare{-x}^{-ml}\pare{\varphi_{4}-\varphi_{2}} + \pare{-x}^{ml}\pare{\varphi_{2}+\varphi_{4}}
\end{pmatrix}
\end{equation*}
Using the boundary conditions, the functions $\pare{-x}^{-ml}\pare{\varphi_{1} + \varphi_{3}}$ and $\pare{-x}^{-ml} \pare{\varphi_{2} - \varphi_{4}}$ go to $0$ at $0$. Using the asymptotic behavior at $0$ of the elements of the domain $D\pare{H_{m}^{s,n}}$ given in \cite{IR14}, we see that $\frac{1}{2}\pare{-x}^{ml}\pare{\varphi_{1}-\varphi_{3}}$ and $\frac{1}{2}\pare{-x}^{ml}\pare{\varphi_{2}+\varphi_{4}}$ admit finite limits that we denote by $c$ and $d$ respectively. Since $V_{\lambda,m}$ is bounded at $0$, we can deduce the integral equation:
\begin{equation*}
\mathcal{M}_{0}\pare{\frac{1}{x}} \varphi\pare{x} = \begin{pmatrix}
c \\
d \\
-c \\
d
\end{pmatrix} + \int_{0}^{x} \mathcal{M}_{0}\pare{\frac{1}{t}} V_{\lambda,m}\pare{t} \varphi\pare{t} dt,
\end{equation*}
which gives:
\begin{equation*}
\varphi\pare{x} = \mathcal{M}_{0}\pare{x} \begin{pmatrix}
c \\
d \\
-c \\
d
\end{pmatrix} + \int_{0}^{x} \mathcal{M}_{0}\pare{x} \mathcal{M}_{0}\pare{\frac{1}{t}} V_{\lambda,m}\pare{t} \varphi\pare{t} dt.
\end{equation*}
Remark that:
\begin{equation*}
\mathcal{M}_{0}\pare{x} \mathcal{M}_{0}\pare{\frac{1}{t}} = \mathcal{M}_{0}\pare{-\frac{x}{t}}.
\end{equation*}
We look for a solution $\varphi$ under the form:
\begin{equation*}
\varphi\pare{x} = \underset{n\geq 0}{\sum} \varphi_{n}\pare{x},
\end{equation*}
with $\varphi_{0}\pare{x} = \mathcal{M}_{0}\pare{x} \begin{pmatrix}
c \\
d \\
-c \\
d
\end{pmatrix}$. We obtain the relations:
\begin{equation*}
\varphi_{n+1}\pare{x} = \int_{0}^{x} \mathcal{M}_{0}\pare{-\frac{x}{t}} V_{\lambda,m}\pare{t} \varphi_{n}\pare{t} dt.
\end{equation*}
We thus have to calculate $\mathcal{M}_{0}\pare{-\frac{x}{t}} V_{\lambda,m}\pare{t}$ where:
\begin{flalign*}
\hphantom{A} & \mathcal{M}_{0}\pare{-\frac{x}{t}} =\\
& \frac{1}{2} \begin{pmatrix}
\pare{\frac{x}{t}}^{ml} + \pare{\frac{x}{t}}^{-ml} & 0 & \pare{\frac{x}{t}}^{ml} - \pare{\frac{x}{t}}^{-ml} & 0 \\
0 & \pare{\frac{x}{t}}^{ml} + \pare{\frac{x}{t}}^{-ml} & 0 & -\pare{\pare{\frac{x}{t}}^{ml} - \pare{\frac{x}{t}}^{-ml}} \\
\pare{\frac{x}{t}}^{ml} - \pare{\frac{x}{t}}^{-ml} & 0 & \pare{\frac{x}{t}}^{ml} + \pare{\frac{x}{t}}^{-ml} & 0 \\
0 & - \pare{\pare{\frac{x}{t}}^{ml} - \pare{\frac{x}{t}}^{-ml}} & 0 & \pare{\frac{x}{t}}^{ml} + \pare{\frac{x}{t}}^{-ml}
\end{pmatrix}
\end{flalign*}
and:
\begin{equation*}
V_{\lambda,m}\pare{t} = \begin{pmatrix}
i \lambda & i \pare{s + \frac{1}{2}} A\pare{t} & -m \pare{B\pare{t} + \frac{l}{t}} & 0 \\
-i \pare{s+ \frac{1}{2}} A\pare{t} & -i\lambda & 0 & m\pare{B\pare{t} + \frac{l}{t}} \\
-m \pare{B\pare{t} + \frac{l}{t}} & 0 & -i \lambda & i\pare{s+ \frac{1}{2}}A\pare{t} \\
0 & m\pare{B\pare{t}+ \frac{l}{t}} & -i\pare{s + \frac{1}{2}} A\pare{t} & i \lambda 
\end{pmatrix}
\end{equation*}
We introduce the matrix:
\begin{equation*}
W_{\lambda,m}\pare{x,t} =  \begin{pmatrix}
- m\pare{B\pare{t} + \frac{l}{t}} & 0 & -i\lambda & i\pare{s+ \frac{1}{2}} A\pare{t} \\
0 & -m\pare{B\pare{t}+\frac{l}{t}} & i\pare{s + \frac{1}{2}} A\pare{t} & -i\lambda \\
i \lambda & i\pare{s+\frac{1}{2}}A\pare{t} & -m\pare{B\pare{t} + \frac{l}{t}} & 0 \\
i\pare{s+\frac{1}{2}}A\pare{t} & i \lambda & 0 & -m\pare{B\pare{t} + \frac{l}{t}} 
\end{pmatrix}
\end{equation*}
Then, tedious calculations lead to:
\begin{flalign*}
\mathcal{M}_{0}\pare{-\frac{x}{t}} V_{\lambda,m}\pare{t} 
&= \frac{1}{2} \left ( \pare{\pare{\frac{x}{t}}^{ml} + \pare{\frac{x}{t}}^{-ml}} V_{\lambda,m}\pare{x,t}\right . \\
& \quad \left .  + \pare{\pare{\frac{x}{t}}^{ml} - \pare{\frac{x}{t}}^{-ml}} W_{\lambda,m}\pare{x,t} \right )
\end{flalign*}
Write $N = 2 \max \pare{{\abs{c},\abs{d}}}$ and $C_{\lambda,m} = \underset{x\in ]-\infty,0[}{\max} \pare{\abs{\lambda},\abs{m\pare{B\pare{x}+\frac{l}{x}}},\abs{A\pare{x}}}$. We will show by induction that for $x\in ]-\infty,0[$, we have:
\begin{equation*}
\abs{\varphi_{n,j}\pare{x}} \leq  N \pare{-x}^{-ml} \frac{1}{n!} \pare{\frac{6 C_{\lambda,m}\pare{s+\frac{1}{2}}}{1-2ml}\pare{-x}}^{n}
\end{equation*}
for all $j=1,\cdots,4$, with $\varphi_{n} = \begin{pmatrix}
\varphi_{n,1} \\
\varphi_{n,2} \\
\varphi_{n,3} \\
\varphi_{n,4} 
\end{pmatrix}$.\\
This is true for the components of $\varphi_{0}$ since:
\begin{equation*}
\varphi_{0}\pare{x} = \begin{pmatrix}
2\pare{-x}^{-ml}c \\
2 \pare{-x}^{-ml}d \\
-2\pare{-x}^{-ml}c \\
2 \pare{-x}^{-ml}d
\end{pmatrix}.
\end{equation*}
Suppose that it is true for the components of $\varphi_{n}$ for some $n \in \N$, then:
\begin{flalign*}
\varphi_{n+1,1}\pare{x}& = \frac{1}{2}\left ( \int_{0}^{x} \pare{i\lambda\pare{\pare{\frac{x}{t}}^{ml} + \pare{\frac{x}{t}}^{-ml}} -m \pare{B\pare{t}+\frac{l}{t}}\pare{\pare{\frac{x}{t}}^{ml}-\pare{\frac{x}{t}}^{-ml}}}\varphi_{n,1}\pare{t} \right . \\
& \quad \left . + i\pare{s+\frac{1}{2}} A\pare{t} \pare{\pare{\frac{x}{t}}^{ml} + \pare{\frac{x}{t}}^{-ml}} \varphi_{n,2}\pare{t} \right . \\
& \quad \left . + \pare{-i\lambda \pare{\pare{\frac{x}{t}}^{ml} - \pare{\frac{x}{t}}^{-ml}} - m \pare{B\pare{t}+\frac{l}{t}}\pare{\pare{\frac{x}{t}}^{ml} + \pare{\frac{x}{t}}^{-ml}}} \varphi_{n,3}\pare{t} \right. \\
& \quad \left . + i\pare{s+\frac{1}{2}}A\pare{t}\pare{\pare{\frac{x}{t}}^{ml} - \pare{\frac{x}{t}}^{-ml}} \varphi_{n,4}\pare{t} dt \right )
\end{flalign*}
Upper bounds for:
\begin{equation*}
\abs{\pare{i\lambda\pare{\pare{\frac{x}{t}}^{ml} + \pare{\frac{x}{t}}^{-ml}} -m \pare{B\pare{t}+\frac{l}{t}}\pare{\pare{\frac{x}{t}}^{ml}-\pare{\frac{x}{t}}^{-ml}}}}
\end{equation*}
and:
\begin{equation*}
i\pare{s+\frac{1}{2}} A\pare{t} \pare{\pare{\frac{x}{t}}^{ml} + \pare{\frac{x}{t}}^{-ml}}
\end{equation*}
are respectively
\begin{equation*}
2C_{\lambda,m}  \pare{\pare{\frac{x}{t}}^{ml} + \pare{\frac{x}{t}}^{-ml}} \hspace{3mm} \text{and} \hspace{3mm} C_{\lambda,m}\pare{s+\frac{1}{2}}  \pare{\pare{\frac{x}{t}}^{ml} + \pare{\frac{x}{t}}^{-ml}}.
\end{equation*}
We obtain:
\begin{flalign*}
\abs{\varphi_{n+1,1}\pare{x}} & \leq 3 C_{\lambda,m} \pare{s+\frac{1}{2}} \int_{x}^{0} \pare{\pare{\frac{x}{t}}^{ml} + \pare{\frac{x}{t}}^{-ml}}\pare{\frac{N}{n!}\pare{\frac{6 C_{\lambda,m}\pare{s+\frac{1}{2}}}{1-2ml}}^{n}\pare{-t}^{n-ml}} dt \\
& \leq \frac{N}{2n!}\frac{\pare{6 C_{\lambda,m}\pare{s+\frac{1}{2}}}^{n+1}}{\pare{1-2ml}^{n}}\pare{\int_{x}^{0} \pare{-x}^{ml}\pare{-t}^{n-2ml} + \pare{-x}^{-ml}\pare{-t}^{n} dt}
\end{flalign*}
This last integral is equal to:
\begin{equation*}
\int_{x}^{0} \pare{-x}^{ml}\pare{-t}^{n-2ml} + \pare{-x}^{-ml}\pare{-t}^{n} dt  = \frac{\pare{-x}^{n+1-ml}}{\pare{n+1}\pare{1-\frac{2ml}{n+1}}} + \frac{\pare{-x}^{n+1-ml}}{n+1}.
\end{equation*}
Since $\frac{2ml}{n+1} \leq 2ml$, we have $\frac{1}{1-\frac{2ml}{n+1}} \leq \frac{1}{1 - 2ml}$. Since $0 < 1-2ml < 1$, $\frac{1}{1-2ml}>1$ and we obtain:
\begin{equation*}
\int_{x}^{0} \pare{-x}^{ml}\pare{-t}^{n-2ml} + \pare{-x}^{-ml}\pare{-t}^{n} dt \leq \frac{2}{\pare{n+1}\pare{1-2ml}} \pare{-x}^{n+1-ml}
\end{equation*}
Consequently:
\begin{equation*}
\abs{\varphi_{n+1,1}\pare{x}} \leq \frac{N \pare{-x}^{-ml}}{\pare{n+1}!}\pare{\frac{6C_{\lambda,m}\pare{s+\frac{1}{2}}}{1-2ml}}^{n+1} \pare{-x}^{n+1}
\end{equation*}
We can do the same with the other coefficients. We deduce that the series $\underset{n\geq 0}{\sum} \varphi_{n}\pare{x}$ converges and that:
\begin{equation*}
\norme{\underset{n\geq 0}{\sum} \varphi_{n}\pare{x}}\leq 4 N \pare{-x}^{-ml} e^{-\frac{6 C_{\lambda,m} \pare{s+\frac{1}{2}}}{1-2ml} x}.
\end{equation*}
Moreover, the boundary conditions are satisfied. Indeed, looking at the expression of $\varphi_{0}$ we see that $\varphi_{0,1}+ \varphi_{0,3} = 0$ and $\varphi_{0,2} - \varphi_{0,4}=0$. By the preceding induction, we know that the norm of $\varphi_{j}$ for all $j \geq 1$ is bounded by a constant times $\pare{-x}^{j-ml}$. Since $ml < \frac{1}{2}$, we deduce that $\frac{\pare{-x}^{j-ml}}{\pare{-x}^{\frac{1}{2}}} = \pare{-x}^{j-\frac{1}{2}-ml} \underset{x\to 0}{\to} 0$ and the boundary conditions are satisfied.\\
We also notice that, at each step, we have a polynomial in $\lambda$ and the convergence of the series is uniform on every compact set so that we obtained a solution which is analytic for $\Im\pare{\lambda} > - \frac{\kappa}{2}$.\\
Finally, we wish to obtain an estimate on the growth of $\varphi$ at $-\infty$. We write $H_{c} = \Gamma^{1}D_{x}$ and $V_{m}\pare{x}= H_{m}^{s,n} - H_{c}$. We have just shown the existence of a solution to:
\begin{equation*}
\partial_{x} \varphi \pare{x} = i\lambda \Gamma^{1} \varphi - i \Gamma^{1} V_{m}\pare{x} \varphi\pare{x}.
\end{equation*}
Denote the value of this solution at $-1$ by $\varphi_{-1}$. Then this solution can be written as:
\begin{equation*}
\varphi \pare{x} = \varphi_{-1} e^{i\lambda \Gamma^{1} \pare{x+1} -i\Gamma^{1} \int_{-1}^{x} V_{m}\pare{t} dt}.
\end{equation*}
For all $x <-1$, we have:
\begin{equation*}
\norme{\varphi \pare{x}} \leq \varphi_{-1}e^{-\abs{\Im\pare{\lambda}}\pare{x+1} + \int_{x}^{-1} \norme{V_{m}\pare{t}}dt}.
\end{equation*}
Since $V_{m}$ is integrable on $]-\infty,-1[$, we have:
\begin{equation*}
\norme{\varphi\pare{x}} \leq \varphi_{-1}e^{-\abs{\Im\pare{\lambda}} +\int_{-\infty}^{-1} \norme{V_{m}\pare{t}}dt} e^{-\abs{\Im\pare{\lambda}}x}.
\end{equation*}
Let $C = \varphi_{-1} e^{-\abs{\Im\pare{\lambda}} +\int_{-\infty}^{-1} \norme{V_{m}\pare{t}}dt}$. Then we obtain the desired estimate.
\end{proof}

We are now interested in the case $2ml\geq1$. The domain of our operator is then $D\pare{H_{m}^{s,n}}=\{\varphi \in \mathcal{H} \lvert H_{m}^{s,n} \varphi \in \mathcal{H}\}$. We have the:
\begin{prop}
Suppose that $2ml \geq 1$ and that $\lambda \in \C$ with $\Im\pare{\lambda}>- \frac{\kappa}{2}$. There exists a solution, analytic for $\Im\pare{\lambda}>- \frac{\kappa}{2}$, to the equation:
\begin{equation*}
H_{m}^{s,n}\varphi = \lambda \varphi
\end{equation*}
going to $0$ as $x$ goes to $0$. Moreover, we have the estimate:
\begin{equation*}
\norme{\varphi\pare{x}}\leq 4 N \pare{-x}^{ml} e^{-6 C_{\lambda,m} \pare{s+\frac{1}{2}} x},
\end{equation*}
where $N$ is a positive constant and $C_{\lambda,m} = \underset{x \in ]-\infty,0[}{\max} \pare{\abs{\lambda},m\abs{B\pare{x}+\frac{l}{x}},\abs{A\pare{x}}}$. Furthermore, we have the same estimate as in the preceding proposition:
\begin{equation*}
\norme{\varphi \pare{x}} \leq C e^{-\abs{\Im\pare{\lambda}}x},
\end{equation*}
as $x$ goes to $-\infty$.
\end{prop}

\begin{proof}
We can do the same argument as in the last proof. We obtain a new equation:
\begin{equation*}
\partial_{x} \pare{\mathcal{M}_{0}\pare{\frac{1}{x}} \varphi\pare{x}} = \mathcal{M}_{0}\pare{\frac{1}{x}} V_{\lambda,m}\pare{x} \varphi\pare{x}.
\end{equation*}
The corresponding integral equation is:
\begin{equation*}
\varphi \pare{x} = \mathcal{M}_{0}\pare{x} \begin{pmatrix}
a \\
b \\
a \\
-b
\end{pmatrix}
+ \int_{0}^{x} \mathcal{M}_{0}\pare{-\frac{x}{t}} V_{\lambda,m}\pare{t}\varphi\pare{t} dt,
\end{equation*}
where $a,b$ are two real constants. We look for $\varphi$ under the form:
\begin{equation*}
\varphi\pare{x} = \underset{n\geq 0}{\sum} \varphi_{n}\pare{x},
\end{equation*}
with:
\begin{equation*}
\varphi_{0}\pare{x} = \mathcal{M}_{0}\pare{x} \begin{pmatrix}
a \\
b \\
a \\
-b
\end{pmatrix} = \begin{pmatrix}
2a \pare{-x}^{ml} \\
2b \pare{-x}^{ml} \\
2a \pare{-x}^{ml} \\
2b \pare{-x}^{ml}
\end{pmatrix},
\end{equation*}
and the recursive equations:
\begin{equation*}
\varphi_{n+1}\pare{x} = \int_{0}^{x} \mathcal{M}_{0}\pare{-\frac{x}{t}} V_{\lambda,m}\pare{t}\varphi_{n}\pare{t} dt.
\end{equation*}
We write $\varphi_{n} = \begin{pmatrix}
\varphi_{n,1} \\
\varphi_{n,2} \\
\varphi_{n,3} \\
\varphi_{n,4}
\end{pmatrix}$, $N = 2\max \pare{\abs{a},\abs{b}}$ and $C_{\lambda,m}$ as in the proposition. We want to show that:
\begin{equation*}
\abs{\varphi_{n,j}\pare{x}} \leq N \pare{-x}^{ml} \frac{1}{n!}\pare{6C_{\lambda,m}\pare{s+\frac{1}{2}}\pare{-x}}^{n},
\end{equation*}
for all $j=1,\cdots,4$. Indeed, this is true for the components of $\varphi_{0}$. Suppose that the components of $\varphi_{n}$ satisfy this estimate. Then, as in the preceding proof, we have:
\begin{equation*}
\abs{\varphi_{n+1,1}\pare{x}} \leq \frac{N}{2n!}\pare{6 C_{\lambda,m}\pare{s+\frac{1}{2}}}^{n+1}\pare{\int_{x}^{0} \pare{-x}^{ml}\pare{-t}^{n} + \pare{-x}^{-ml}\pare{-t}^{n+2ml} dt}
\end{equation*}
This integral is equal to:
\begin{equation*}
\int_{x}^{0} \pare{-x}^{ml}\pare{-t}^{n} + \pare{-x}^{-ml}\pare{-t}^{n+2ml} dt  = \frac{\pare{-x}^{n+1+ml}}{n+1} + \frac{\pare{-x}^{n+1+ml}}{n+1+2ml}.
\end{equation*}
Since $n+1+2ml \geq n+1$, we deduce that:
\begin{equation*}
\abs{\varphi_{n+1,1}\pare{x}} \leq N \pare{-x}^{ml} \frac{1}{\pare{n+1}!}\pare{6C_{\lambda,m}\pare{s+\frac{1}{2}}\pare{-x}}^{n+1}.
\end{equation*}
As before, at each step, we have a polynomial in $\lambda$ so that we obtained an analytic function. The last estimate follows from the same argument as in the preceding proof.
\end{proof}

\section{Resolvent formula}

In this section, we denote by $\psi$ a Jost solution corresponding to $\varphi_{3}$ and $\varphi$ a solution satisfying the boundary conditions. We look for a solution $u$ of:
\begin{equation*}
\pare{H_{m}^{s,n} - \lambda}u = f
\end{equation*}
with $f \in L^{2}\pare{]-\infty,0[}$ for $\Im\pare{\lambda}>0$.\\
We introduce $\tilde{\psi} = \pare{-i} \gamma^{0} \gamma^{1} \gamma^{2} \psi$ and $\tilde{\varphi}  = \pare{-i} \gamma^{0} \gamma^{1} \gamma^{2} \varphi$ where $\gamma^{0}$, $\gamma^{1}$, $\gamma^{2}$ are the Dirac matrices \eqref{MatDir}. We also write $\alpha = \varphi_{1} \psi_{2} - \psi_{1} \varphi_{2} + \varphi_{3} \psi_{4} - \psi_{3} \varphi_{4}$, $\beta = \varphi_{1} \psi_{3} - \psi_{1} \varphi_{3} + \varphi_{2} \psi_{4} - \psi_{2} \varphi_{4}$ and:
\begin{equation}\label{Mabexpr}
M_{\alpha,\beta} =\begin{pmatrix}
0 & \alpha & \beta & 0 \\
-\alpha & 0 & 0 & \beta \\
-\beta & 0 & 0 & \alpha \\
0 &  -\beta & -\alpha & 0 
\end{pmatrix}.
\end{equation}
We begin by a lemma about this matrix:
\begin{lem}
The functions $\alpha$ and $\beta$ are independant of $x$. Moreover, the matrix $M_{\alpha,\beta}$ is invertible.
\end{lem}

\begin{proof}
We first use the fact that $\varphi$, $\psi$, $\tilde{\varphi}$ and $\tilde{\psi}$ satisfy $H_{m}^{s,n} \varphi = \lambda \varphi$, $H_{m}^{s,n}\psi= \lambda \psi$, $H_{m}^{s,n} \tilde{\varphi} = \lambda \tilde{\varphi}$ and $H_{m}^{s,n} \tilde{\psi} = \lambda \tilde{\psi}$. For the coordinates of $\varphi = \begin{pmatrix}
\varphi_{1} \\
\varphi_{2} \\
\varphi_{3} \\
\varphi_{4} 
\end{pmatrix}$, we obtain the equations:
\begin{flalign*}
\partial_{x} \varphi_{1} \pare{x} & = i\lambda \varphi_{1} \pare{x} + i\pare{s+\frac{1}{2}}A\pare{x} \varphi_{2}\pare{x} - mB\pare{x} \varphi_{3}\pare{x}, \\
\partial_{x} \varphi_{2} \pare{x} & = -i\lambda \varphi_{2} \pare{x} -i\pare{s+\frac{1}{2}}A\pare{x} \varphi_{1}\pare{x} + mB\pare{x} \varphi_{4} \pare{x}, \\
\partial_{x} \varphi_{3}\pare{x} & = -i\lambda \varphi_{3}\pare{x} + i\pare{s+\frac{1}{2}}A\pare{x} \varphi_{4}\pare{x} - mB\pare{x} \varphi_{1}\pare{x}, \\
\partial_{x}\varphi_{4} \pare{x} & = i\lambda \varphi_{4}\pare{x} - i \pare{s+\frac{1}{2}}A\pare{x} \varphi_{3}\pare{x} + mB\pare{x} \varphi_{2}\pare{x}.
\end{flalign*}
The same equations are satisfied for the other solutions. Then we can calculate:
\begin{flalign*}
\partial_{x} \alpha  & = \partial_{x}\pare{\varphi_{1}} \psi_{2} + \varphi_{1} \partial_{x}\pare{\psi_{2}} - \partial_{x}\pare{\psi_{1}} \varphi_{2} - \psi_{1} \partial_{x} \pare{\varphi_{2}} + \partial_{x}\pare{\varphi_{3}}\psi_{4} + \varphi_{3} \partial_{x} \pare{\psi_{4}} \\
& \quad - \partial_{x} \pare{\psi_{3}}\varphi_{4} - \psi_{3}\partial_{x} \varphi_{4} 
\end{flalign*}
\begin{flalign*}
\hphantom{A} & = \pare{i\lambda \varphi_{1}  + i\pare{s+\frac{1}{2}}A \varphi_{2} - mB \varphi_{3}}\psi_{2} + \varphi_{1} \pare{-i\lambda \psi_{2} -i\pare{s+\frac{1}{2}}A \psi_{1} + mB \psi_{4}} \\
& \quad - \pare{i\lambda \psi_{1}  + i\pare{s+\frac{1}{2}}A \psi_{2} - mB \psi_{3}}\varphi_{2} - \psi_{1}\pare{-i\lambda\varphi_{2} - i\pare{s+\frac{1}{2}}A\varphi_{1} + mB\varphi_{4}} \\
& \quad + \pare{-i\lambda \varphi_{3} + i\pare{s+\frac{1}{2}} A \varphi_{4} - m B \varphi_{1}}\psi_{4} + \varphi_{3} \pare{i\lambda\psi_{4} - i\pare{s+\frac{1}{2}}A \psi_{3} + m B\psi_{2}} 
\end{flalign*}
\begin{flalign*}
\hphantom{A} & \quad - \pare{-i\lambda \psi_{3} + i\pare{s+\frac{1}{2}} A \psi_{4} - mB \psi_{1}} \varphi_{4} - \psi_{3} \pare{i\lambda \varphi_{4} - i\pare{s+\frac{1}{2}}A \varphi_{3} + mB\varphi_{2}} \\
& = \pare{\varphi_{1} \psi_{2} -  \varphi_{1} \psi_{2} - \psi_{1}\varphi_{2} + \psi_{1}\varphi_{2} - \varphi_{3}\psi_{4} + \varphi_{3} \psi_{4} +\psi_{3}\varphi_{4} - \psi_{3}\varphi_{4}}i\lambda \\
& \quad + \pare{\varphi_{2}\psi_{2} -\varphi_{1}\psi_{1} - \psi_{2} \varphi_{2} + \psi_{1}\varphi_{1} + \varphi_{4} \psi_{4} - \varphi_{3} \psi_{3} - \psi_{4}\varphi_{4} + \psi_{3}\varphi_{3}} i\pare{s+\frac{1}{2}}A\\
& \quad + \pare{- \varphi_{3} \psi_{2} + \varphi_{1} \psi_{4} + \psi_{3}\varphi_{2} - \psi_{1}\varphi_{4} - \varphi_{1}\psi_{4} + \varphi_{3} \psi_{2} + \psi_{1}\varphi_{4} - \psi_{3} \varphi_{2}}mB \\
& = 0
\end{flalign*} 
which shows that $\alpha$ is independant of $x$. A similar calculation holds for $\beta$.\\
Concerning the invertibility of $M_{\alpha,\beta}$, its determinant is given by:
\begin{equation*}
\det \pare{M_{\alpha,\beta}}  = \pare{\pare{\alpha-\beta}\pare{\alpha +\beta}}^{2}.
\end{equation*}
This matrix is thus not invertible if $\alpha = \beta$ or $\alpha = -\beta$. Suppose, for example, that $\alpha = \beta$. We write:
\begin{flalign*}
T_{m}\pare{\lambda} f \pare{x}& = \int_{-\infty}^{x} \pare{\varphi\pare{x} \psi^{t}\pare{y} + \tilde{\varphi}\pare{x} \tilde{\psi}^{t}\pare{y}}f\pare{y} dy \\
& \quad + \int_{x}^{0}\pare{\psi\pare{x} \varphi^{t}\pare{y} + \tilde{\psi}\pare{x} \tilde{\varphi}^{t}\pare{y}}f\pare{y} dy.
\end{flalign*}
$T_{m}\pare{\lambda}$ is an operator with kernel:
\begin{equation*}
T_{m}\pare{x,y,\lambda}= \pare{\varphi\pare{x} \psi^{t}\pare{y} + \tilde{\varphi}\pare{x} \tilde{\psi}^{t}\pare{y}}\mathds{1}_{]-\infty,x[}\pare{y} + \pare{\psi\pare{x}\varphi^{t}\pare{y} + \tilde{\psi}\pare{x} \tilde{\varphi}^{t}\pare{y}}\mathds{1}_{]x,0[}\pare{y}.
\end{equation*}
Using again that $\varphi$, $\psi$, $\tilde{\varphi}$ and $\tilde{\psi}$ are generalized eigenvector of $H_{m}^{s,n}$ for the eigenvalue $\lambda$, we get:
\begin{equation*}
H_{m}^{s,n} \pare{T_{m}\pare{\lambda}f} \pare{x}  = \lambda T_{m}\pare{\lambda}f\pare{x} - i\Gamma^{1} M_{\alpha,\alpha} f\pare{x}.
\end{equation*}
We can choose a function $v \in L^{2}\pare{]-\infty,0[}$ and take $f = v\begin{pmatrix}
0 \\
1 \\
-1 \\
0 
\end{pmatrix}$ such that:
\begin{equation*}
M_{\alpha,\alpha}f = \begin{pmatrix}
0 & \alpha & \alpha & 0 \\
-\alpha & 0 & 0 & \alpha \\
-\alpha & 0 & 0 & \alpha \\
0 &  -\alpha & -\alpha & 0 
\end{pmatrix} f = 0.
\end{equation*}
We thus would have an eigenvector for the eigenvalue $\lambda$ which is impossible since there's no eigenvalue for our operator (by proposition $3.11$ in \cite{IR14}). We can do the same when $\alpha = - \beta$. Thus $M_{\alpha,\beta}$ is invertible.
\end{proof}
Now, consider the function defined by:
\begin{flalign*}
R_{m}^{s,n}\pare{x,y,\lambda} & = \pare{\varphi\pare{x} \psi^{t}\pare{y} + \tilde{\varphi}\pare{x} \tilde{\psi}^{t}\pare{y}} M_{\alpha,\beta}^{-1} i\Gamma^{1} \mathds{1}_{]-\infty,x[}\pare{y} \\
& \quad + \pare{\psi\pare{x}\varphi^{t}\pare{y} + \tilde{\psi}\pare{x} \tilde{\varphi}^{t}\pare{y}}M_{\alpha,\beta}^{-1} i\Gamma^{1}\mathds{1}_{]x,0[}\pare{y}
\end{flalign*}
for $\Im\pare{\lambda} >0$. We define the corresponding integral operator:
\begin{equation*}
R_{m}^{s,n}\pare{\lambda}f \pare{x} = \int_{-\infty}^{0} R_{m}^{s,n}\pare{x,y,\lambda}f\pare{y} dy.
\end{equation*}
We first show the boundedness of this operator:
\begin{lem}\label{BDEDRES}
The operator $R_{m}^{s,n}\pare{\lambda}$ is bounded from $\mathcal{H}_{s,n}$ into itself for any $m>0$ and $\lambda$ such that $\Im\pare{\lambda}>0$.
\end{lem}

\begin{proof}
We have seen in \ref{propJost} that, for all $x\in ]-\infty,0[$, we have the estimate:
\begin{equation*}
\norme{\psi\pare{x}} \leq e^{\Im\pare{\lambda} x} e^{\int_{-\infty}^{x} \norme{V_{m}\pare{t}} dt}.
\end{equation*}
Let $\epsilon >0$. We deduce that, for all $x\in ]-\infty,-\epsilon[$, we have:
\begin{equation*}
\norme{\psi\pare{x}} \leq C_{m,\epsilon} e^{\Im\pare{\lambda} x}.
\end{equation*}
Moreover, for all $x \in ]-\epsilon,0[$, we have:
\begin{flalign*}
\norme{\psi\pare{x}} & \leq e^{\Im\pare{\lambda} x} e^{\int_{-\infty}^{-\epsilon} \norme{V_{m}\pare{t}} dt + \int_{-\epsilon}^{x} \norme{V_{m}\pare{t}} dt} \\
& \leq C_{\epsilon,m} e^{\int_{-\epsilon}^{x} \abs{\frac{ml}{t}} dt}.
\end{flalign*}
Since:
\begin{equation*}
\int_{-\epsilon}^{x} \abs{\frac{ml}{t}} dt =- ml \int_{-\epsilon}^{x} \frac{1}{t} dt = -ml \left [\ln \pare{-t} \right]_{-\epsilon}^{x} = \ln \pare{\pare{\frac{\epsilon}{-x}}^{ml}},
\end{equation*}
we deduce that:
\begin{equation*}
\norme{\psi\pare{x}} \leq C_{\epsilon,m} \pare{-x}^{-ml}
\end{equation*}
on $]-\epsilon,0[$. Thus $\psi$ is in $L^{2}\pare{]-\infty,0[}$ for $2ml<1$ (but does not satisfy the boundary conditions, otherwise, it would be an eigenvector). Recall that, for $2ml <1$, we have the following estimates:
\begin{flalign*}
\norme{\varphi\pare{x}} &\leq 4 N \pare{-x}^{-ml} e^{-6 C_{\lambda,m} \pare{s+\frac{1}{2}} x},\\
\norme{\varphi \pare{x}} & \leq C e^{-\Im\pare{\lambda}x},
\end{flalign*}
where the first inequality is taken near $0$ and the second one at $-\infty$. In the case $2ml\geq 1$, we have the estimates:
\begin{flalign*}
\norme{\varphi\pare{x}} &\leq 4 N \pare{-x}^{ml} e^{-6 C_{\lambda,m} \pare{s+\frac{1}{2}} x}, \\
\norme{\varphi \pare{x}} & \leq C e^{-\Im\pare{\lambda}x}.
\end{flalign*}
Recall that:
\begin{flalign*}
R_{m}^{s,n}\pare{x,y,\lambda} & = \pare{\varphi\pare{x} \psi^{t}\pare{y} + \tilde{\varphi}\pare{x} \tilde{\psi}^{t}\pare{y}} M_{\alpha,\beta}^{-1} i\Gamma^{1} \mathds{1}_{]-\infty,x[}\pare{y} \\
& \quad + \pare{\psi\pare{x}\varphi^{t}\pare{y} + \tilde{\psi}\pare{x} \tilde{\varphi}^{t}\pare{y}}M_{\alpha,\beta}^{-1} i\Gamma^{1}\mathds{1}_{]x,0[}\pare{y}.
\end{flalign*}
We are first interested in the case $2ml \geq 1$.\\
We remark that, in this formula, $x$ and $y$ have a symmetrical role. Indeed, $\mathds{1}_{]-\infty,x[}\pare{y} = \mathds{1}_{]y,0[}\pare{x}$ and $\mathds{1}_{]x,0[}\pare{y} = \mathds{1}_{]-\infty,y[}\pare{x}$. We thus concentrate on $\int_{-\infty}^{0} R_{m}^{s,n}\pare{x,y,\lambda} dy$:
\begin{flalign*}
\int_{-\infty}^{0} R_{m}^{s,n}\pare{x,y,\lambda} dy & = \int_{-\infty}^{x} \pare{\varphi\pare{x} \psi^{t}\pare{y} + \tilde{\varphi}\pare{x} \tilde{\psi}^{t}\pare{y}} M_{\alpha,\beta}^{-1} i\Gamma^{1} dy \\
& \quad + \int_{x}^{0}  \pare{\psi\pare{x}\varphi^{t}\pare{y} + \tilde{\psi}\pare{x} \tilde{\varphi}^{t}\pare{y}}M_{\alpha,\beta}^{-1} i\Gamma^{1} dy.
\end{flalign*}
Since $\psi$ is integrable on $]-\infty,x[$ and $\varphi$ is integrable on $]x,0[$, this integral is well-defined and bounded for all $x$ in a compact subset of $]-\infty,0[$. We study the limits as $x$ goes to $-\infty$ and $0$ of the two preceding integrals. We have:
\begin{equation*}
\norme{\int_{-\infty}^{x} \psi^{t} \pare{y} dy}  \leq \int_{-\infty}^{x} 2 e^{\Im\pare{\lambda} y} dy = \frac{2}{\Im\pare{\lambda}}e^{\Im\pare{\lambda} x}.
\end{equation*}
Indeed, $\int_{-\infty}^{x} \norme{V_{m}\pare{t}} dt$ tends to $0$ at $-\infty$. Consequently, $e^{\int_{-\infty}^{x} \norme{V_{m}\pare{t}} dt} \leq 2$ for $x$ going to $-\infty$. Using the estimates on $\varphi$, we deduce that $\int_{-\infty}^{x} \pare{\varphi\pare{x} \psi^{t}\pare{y} + \tilde{\varphi}\pare{x} \tilde{\psi}^{t}\pare{y}} M_{\alpha,\beta}^{-1} i\Gamma^{1} dy$ is bounded at $-\infty$.\\
Let $A >0$ such that $\norme{\varphi \pare{x}} \leq C e^{-\Im\pare{\lambda}x}$ for $x<A$. Then, for all $x<A$, we have:
\begin{flalign*}
\norme{\int_{x}^{0}  \varphi^{t}\pare{y}dy} 
& \leq \int_{x}^{A} C e^{-\Im\pare{\lambda} y} dy + \int_{A}^{0} \norme{\varphi^{t}\pare{y}}dy \\
& = \frac{C}{\Im\pare{\lambda}} \pare{e^{-\Im\pare{\lambda}x} - e^{-\Im\pare{\lambda}A}} + \int_{A}^{0} \norme{\varphi^{t}\pare{y}}dy.
\end{flalign*}
Since $\norme{\psi\pare{x}} \leq e^{\Im\pare{\lambda}x}$ at $-\infty$, we obtain
\begin{equation*}
\int_{x}^{0}  \pare{\psi\pare{x}\varphi^{t}\pare{y} + \tilde{\psi}\pare{x} \tilde{\varphi}^{t}\pare{y}}M_{\alpha,\beta}^{-1} i\Gamma^{1} dy
\end{equation*}
is bounded at $-\infty$.\\
We now look at $0$. We consider $x \in ]-\infty,0[$ sufficiently close to $0$. We have:
\begin{equation*}
\norme{\int_{x}^{0} \varphi^{t}\pare{y} M_{\alpha,\beta}^{-1} i\Gamma^{1} dy}  \leq \int_{x}^{0} C \pare{-y}^{ml} dy 
 = C \left[-\frac{\pare{-y}^{1+ml}}{1+ml} \right]_{x}^{0} = C\pare{-x}^{1+ml}.
\end{equation*}
Since $\norme{\psi \pare{x}} \leq C \pare{-x}^{-ml}$, we obtain:
\begin{equation*}
\int_{x}^{0}  \pare{\psi\pare{x}\varphi^{t}\pare{y} + \tilde{\psi}\pare{x} \tilde{\varphi}^{t}\pare{y}}M_{\alpha,\beta}^{-1} i\Gamma^{1} dy \leq 2C \pare{-x},
\end{equation*}
which is bounded at $0$. Now, let $\epsilon >0$ sufficiently small. We have:
\begin{flalign*}
\norme{\int_{-\infty}^{x} \psi^{t}\pare{y} M_{\alpha,\beta}^{-1} i\Gamma^{1} dy} & \leq \int_{-\infty}^{-\epsilon} \norme{\psi\pare{y}} dy + \int_{-\epsilon}^{x} \norme{\psi\pare{y}} dy \\
& \leq \int_{-\infty}^{-\epsilon} \norme{\psi\pare{y}} dy + \int_{-\epsilon}^{x} C \pare{-y}^{-ml} dy \\
& = \begin{cases}
\int_{-\infty}^{-\epsilon} \norme{\psi\pare{y}} dy + C \pare{\frac{\pare{\epsilon}^{1-ml}}{1-ml} - \frac{\pare{-x}^{1-ml}}{1-ml}}, \hspace{2mm} \text{if} \hspace{2mm} ml \neq 1 \\
\int_{-\infty}^{-\epsilon} \norme{\psi\pare{y}} dy + C \pare{\ln\pare{\epsilon} - \ln \pare{-x}}, \hspace{2mm} \text{if} \hspace{2mm} ml=1
\end{cases}
\end{flalign*}
Since $\norme{\varphi\pare{x}} \leq C \pare{-x}^{ml}$, we deduce that:
\begin{flalign*}
\hphantom{A} &\norme{\int_{-\infty}^{x} \pare{\varphi\pare{x} \psi^{t}\pare{y} + \tilde{\varphi}\pare{x} \tilde{\psi}^{t}\pare{y}} M_{\alpha,\beta}^{-1} i\Gamma^{1} dy} \\
&\leq \begin{cases}
2 \pare{-x}^{ml} \pare{\int_{-\infty}^{-\epsilon} \norme{\psi\pare{y}} dy + C \pare{\frac{\pare{\epsilon}^{1-ml}}{1-ml} - \frac{\pare{-x}^{1-ml}}{1-ml}}}, \hspace{2mm} \text{if} \hspace{2mm} ml \neq 1 \\
2 \pare{-x}^{ml} \pare{\int_{-\infty}^{-\epsilon} \norme{\psi\pare{y}} dy + C \pare{\ln\pare{\epsilon} - \ln \pare{-x}}}, \hspace{2mm} \text{if} \hspace{2mm} ml=1.
\end{cases}
\end{flalign*}
This proves that $\int_{-\infty}^{x} \pare{\varphi\pare{x} \psi^{t}\pare{y} + \tilde{\varphi}\pare{x} \tilde{\psi}^{t}\pare{y}} M_{\alpha,\beta}^{-1} i\Gamma^{1} dy$ is bounded at $0$. We can now apply the Schur's lemma which proves that $R_{m}^{s,n}\pare{\lambda}$ is bounded from $L^{2}\pare{]-\infty,0[}$ into itself.\\
We now study the case $2ml <1$. In this case, $\varphi$ is integrable at $0$ but is not bounded. The preceding argument does not work at $0$. Recall that:
\begin{flalign*}
R_{m}^{s,n}\pare{\lambda}f \pare{x} & = \int_{-\infty}^{x} \pare{\varphi\pare{x} \psi^{t}\pare{y} + \tilde{\varphi}\pare{x} \tilde{\psi}^{t}\pare{y}} M_{\alpha,\beta}^{-1} i\Gamma^{1} f\pare{y} dy \\
& \quad + \int_{x}^{0}  \pare{\psi\pare{x}\varphi^{t}\pare{y} + \tilde{\psi}\pare{x} \tilde{\varphi}^{t}\pare{y}}M_{\alpha,\beta}^{-1} i\Gamma^{1} f\pare{y} dy,
\end{flalign*}
and that $\psi$ is a $L^{2}$ function which does not satisfies the boundary conditions. Let $\epsilon >0$. We calculate:
\begin{flalign*}
\norme{\mathds{1}_{]-\epsilon,0[} R_{m}^{s,n}\pare{\lambda}f}_{\mathcal{H}_{m}^{s,n}}^{2} & =\int_{-\epsilon}^{0} \left \lVert \int_{-\infty}^{x} \pare{\varphi\pare{x} \psi^{t}\pare{y} + \tilde{\varphi}\pare{x} \tilde{\psi}^{t}\pare{y}} M_{\alpha,\beta}^{-1} i\Gamma^{1} f\pare{y} dy \right .\\
& \quad \left . + \int_{x}^{0}  \pare{\psi\pare{x}\varphi^{t}\pare{y} + \tilde{\psi}\pare{x} \tilde{\varphi}^{t}\pare{y}}M_{\alpha,\beta}^{-1} i\Gamma^{1} f\pare{y} dy \right \rVert ^{2} dx \\
& \leq 4 \left ( \int_{-\epsilon}^{0} \left ( \int_{-\infty}^{x} \norme{\varphi\pare{x}} \norme{\psi^{t}\pare{y}} \norme{ f\pare{y}} dy \right)^{2} \right . \\
& \quad \left . + \pare{\int_{x}^{0} \norme{\psi\pare{x}} \norme{ \varphi^{t}\pare{y}} \norme{ f\pare{y}} dy}^{2} dx \right) \\
& \leq 8\norme{\varphi}_{\left [L^{2}\pare{]-\epsilon,0[}\right ]^{4}}^{2} \norme{\psi}_{\mathcal{H}_{m}^{s,n}}^{2} \norme{f}_{\mathcal{H}_{m}^{s,n}}^{2}
\end{flalign*}
using the Cauchy-Schwarz inequality. Moreover, we have:
\begin{flalign*}
\norme{\mathds{1}_{]-\infty,-\epsilon[} R_{m}^{s,n}\pare{\lambda} f}_{L^{2}\pare{]-\infty,0[}}^{2} & \leq 4 \left ( \int_{-\infty}^{-\epsilon} \left ( \int_{-\infty}^{x} \norme{\varphi\pare{x}} \norme{\psi^{t}\pare{y}} \norme{ f\pare{y}} dy \right)^{2} \right . \\
& \quad \left . + \pare{\int_{x}^{0} \norme{\psi\pare{x}} \norme{ \varphi^{t}\pare{y}} \norme{ f\pare{y}} dy}^{2} dx \right).
\end{flalign*} 
We study the first term of the last sum:
\begin{flalign*}
\hphantom{A} & \int_{-\infty}^{-\epsilon} \left ( \int_{-\infty}^{x} \norme{\varphi\pare{x}} \norme{\psi^{t}\pare{y}} \norme{ f\pare{y}} dy  \right)^{2} dx  \\
& \leq \int_{-\infty}^{-\epsilon} \pare{\int_{-\infty}^{x} \norme{\varphi\pare{x}} \norme{\psi\pare{y}} dy } \pare{\int_{-\infty}^{x} \norme{\varphi\pare{x}}\norme{\psi\pare{y}} \norme{f\pare{y}}^{2} dy}dx
\end{flalign*}
As in the case $2ml \geq 1$, we can show that $\int_{-\infty}^{x} \norme{\varphi\pare{x}} \norme{\psi\pare{y}} dy $ is bounded on $]-\infty,-\epsilon[$. Furthermore:
\begin{flalign*}
\hphantom{A} & \int_{-\infty}^{-\epsilon} \int_{-\infty}^{x} \norme{\varphi\pare{x}}\norme{\psi\pare{y}} \norme{f\pare{y}}^{2} dy dx \\
& = \int_{-\infty}^{0} \mathds{1}_{]-\infty,-\epsilon[}\pare{y} \norme{f\pare{y}}^{2} \norme{\psi\pare{y}} \pare{\int_{y}^{0} \norme{\varphi\pare{x}}dx } dy
\end{flalign*}
Since $\varphi$ is integrable at $0$, $\norme{\psi\pare{y}} \pare{\int_{y}^{0} \norme{\varphi\pare{x}}dx }$ is bounded at $-\epsilon$. Thanks to the decay of $\psi$ at $-\infty$, we can show, as in the case $2ml \geq 1$, that this term is bounded at $-\infty$. We obtain:
\begin{equation*}
\int_{-\infty}^{-\epsilon} \left ( \int_{-\infty}^{x} \norme{\varphi\pare{x}} \norme{\psi^{t}\pare{y}} \norme{ f\pare{y}} dy \right)^{2} \leq C_{\epsilon} \norme{f}_{L^{2}\pare{]-\infty,0[}}^{2}.
\end{equation*}
We are now interested in the second term and we have:
\begin{flalign*}
\int_{-\infty}^{-\epsilon} \pare{\int_{x}^{0} \norme{\psi\pare{x}} \norme{ \varphi^{t}\pare{y}} \norme{ f\pare{y}} dy}^{2} & \leq 2\left ( \int_{-\infty}^{-\epsilon} \pare{\int_{x}^{-\epsilon} \norme{\psi\pare{x}}\norme{\varphi\pare{y}}\norme{f\pare{y}} dy}^{2} \right . \\
& \quad \left . + \pare{\int_{-\epsilon}^{0} \norme{\psi\pare{x}}\norme{\varphi\pare{y}}\norme{f\pare{y}} dy}^{2}dx \right).
\end{flalign*}
The second term is bounded by:
\begin{equation*}
\int_{-\infty}^{-\epsilon} \pare{\int_{-\epsilon}^{0} \norme{\psi\pare{x}}\norme{\varphi\pare{y}}\norme{f\pare{y}} dy}^{2} \leq \norme{\psi}_{\mathcal{H}_{m}^{s,n}}^{2} \norme{\varphi}_{\left [ L^{2}\pare{]-\epsilon,0[} \right ]^{4}}^{2} \norme{f}_{\mathcal{H}_{m}^{s,n}}^{2}.
\end{equation*}
We also have:
\begin{flalign*}
\hphantom{A} & \int_{-\infty}^{-\epsilon} \pare{\int_{x}^{-\epsilon} \norme{\psi\pare{x}}\norme{\varphi\pare{y}}\norme{f\pare{y}} dy}^{2} \\
& \leq \int_{-\infty}^{-\epsilon} \pare{\int_{x}^{-\epsilon} \norme{\psi\pare{x}} \norme{\varphi\pare{y}} dy} \pare{\int_{x}^{-\epsilon} \norme{\psi\pare{x}}\norme{\varphi\pare{y}}\norme{f\pare{y}}^{2} dy} dx.
\end{flalign*}
On $]-\infty,-\epsilon[$, $\int_{x}^{-\epsilon} \norme{\psi\pare{x}} \norme{\varphi\pare{y}} dy$ is bounded. We have to find a bound on:
\begin{equation*}
 \int_{-\infty}^{-\epsilon} \int_{x}^{-\epsilon} \norme{\psi\pare{x}}\norme{\varphi\pare{y}}\norme{f\pare{y}}^{2} dy dx  = \int_{-\infty}^{-\epsilon} \norme{f\pare{y}}^{2} \norme{\varphi\pare{y}} \pare{\int_{-\infty}^{y} \norme{\psi\pare{x}}dx} dy.
\end{equation*}
As before, $\norme{\varphi\pare{y}} \pare{\int_{-\infty}^{y} \norme{\psi\pare{x}}dx}$ is bounded on $]-\infty,-\epsilon[$. We obtain:
\begin{equation*}
\int_{-\infty}^{-\epsilon} \pare{\int_{x}^{0} \norme{\psi\pare{x}} \norme{ \varphi^{t}\pare{y}} \norme{ f\pare{y}} dy}^{2} \leq C_{\epsilon} \norme{f}_{\mathcal{H}_{m}^{s,n}}^{2}.
\end{equation*}
Consequently:
\begin{equation*}
\norme{\mathds{1}_{]-\infty,-\epsilon[} R_{m}^{s,n}\pare{\lambda} f}_{\mathcal{H}_{m}^{s,n}}^{2} \leq C_{\epsilon} \norme{f}_{\mathcal{H}_{m}^{s,n}}^{2}.
\end{equation*}
We deduce that $R_{m}^{s,n}\pare{\lambda}$ is a bounded operator on $ L^{2}\pare{]-\infty,0[}$.

\end{proof}
Moreover, we can show that the boundary conditions are satisfied:
\begin{lem}
For all $f\in L^{2}\pare{]-\infty,0[}$, all $\lambda \in \C$ such that $\Im\pare{\lambda} >0$, we have:
\begin{enumerate}
\item[-] If $2ml \geq 1$, then $R_{m}^{s,n}\pare{\lambda}f \pare{x}$ tends to $0$ at $0$.
\item[-] If $2ml<\frac{1}{2}$, then $\norme{\pare{\gamma^{1}+i} R_{m}^{s,n}\pare{\lambda}f\pare{x}} = O\pare{\sqrt{-x}}$ at $0$.
\end{enumerate}
\end{lem}

\begin{proof}
Recall that:
\begin{flalign*}
R_{m}^{s,n}\pare{\lambda}f \pare{x} & = \int_{-\infty}^{x} \pare{\varphi\pare{x} \psi^{t}\pare{y} + \tilde{\varphi}\pare{x} \tilde{\psi}^{t}\pare{y}} M_{\alpha,\beta}^{-1} i\Gamma^{1} f\pare{y} dy \\
& \quad + \int_{x}^{0}  \pare{\psi\pare{x}\varphi^{t}\pare{y} + \tilde{\psi}\pare{x} \tilde{\varphi}^{t}\pare{y}}M_{\alpha,\beta}^{-1} i\Gamma^{1} f\pare{y} dy.
\end{flalign*}
In the case $2ml \geq 1$, when $x$ goes to $0$, we have:
\begin{flalign*}
\norme{\int_{x}^{0}  \pare{\psi\pare{x}\varphi^{t}\pare{y} + \tilde{\psi}\pare{x} \tilde{\varphi}^{t}\pare{y}}M_{\alpha,\beta}^{-1} i\Gamma^{1} f\pare{y} dy} & \leq 2 \norme{\psi\pare{x}} \norme{\varphi}_{\left [ L^{2}\pare{]x,0[} \right ]^{4}} \norme{f}_{\mathcal{H}_{m}^{s,n}} \\
& \leq C \pare{-x}^{-ml} \pare{-x}^{\frac{1}{2}+ ml} \norme{f}_{\mathcal{H}_{m}^{s,n}}.
\end{flalign*}
Indeed:
\begin{equation*}
\int_{x}^{0}\norme{\varphi\pare{y}}^{2} dy \leq C \left [ -\frac{ \pare{-x}^{1+2ml}}{1+2ml} \right ].
\end{equation*}
We deduce that the second term in the expression of $R_{m}^{s,n}\pare{\lambda}f $ goes to $0$. Moreover, let $\epsilon >0$, we have:
\begin{flalign*}
\norme{\int_{-\infty}^{x} \pare{\varphi\pare{x} \psi^{t}\pare{y} + \tilde{\varphi}\pare{x} \tilde{\psi}^{t}\pare{y}} M_{\alpha,\beta}^{-1} i\Gamma^{1} f\pare{y} dy} & \leq 2 \norme{\varphi\pare{x}} \left (\int_{-\infty}^{-\epsilon} \norme{\psi\pare{y}} \norme{f\pare{y}} dy \right . \\
& \quad \left . + \int_{-\epsilon}^{x} \norme{\psi\pare{y}} \norme{f\pare{y}} dy \right ) \\
& \leq 2 C \pare{-x}^{ml} \left ( \norme{\psi}_{\left [ L^{2}\pare{]-\infty,-\epsilon[} \right ]^{4}} \norme{f}_{\mathcal{H}_{m}^{s,n}} \right . \\
& \quad \left . + \pare{\left [ - \frac{\pare{-y}^{1-2ml}}{1-2ml} \right ]_{-\epsilon}^{x}}^{\frac{1}{2}} \norme{f}_{\mathcal{H}_{m}^{s,n}} \right )
\end{flalign*}
when $2ml \neq 1$. The term on the right hand side goes to $0$. In the case $2ml = 1$, we replace $\frac{\pare{-y}^{1-2ml}}{1-2ml}$ by $\ln\pare{-y}$ and we have the same result.\\
In the case $2ml<1$, by the Cauchy-Schwarz inequality, we have:
\begin{equation*}
\norme{\int_{-\infty}^{x} \psi^{t}\pare{y}f\pare{y} dy}  \leq \norme{\psi}_{L^{2}\pare{]-\infty,x[}}\norme{f}_{L^{2}\pare{]-\infty,0[}}.
\end{equation*}
Using the behavior at $0$ of $\psi$, we have:
\begin{equation*}
\norme{\psi}_{L^{2}\pare{]-\infty,x[}} 
 \leq \norme{\psi}_{L^{2}\pare{]-\infty,-\epsilon[}} + \frac{1}{\pare{1-2ml}^{\frac{1}{2}}}\pare{\pare{-x}^{\frac{1}{2}-ml} + \pare{-\epsilon}^{\frac{1}{2} -ml}}.
\end{equation*}
Since $\varphi$ satisfies the boundary conditions $\norme{\pare{\gamma^{1}+i}\varphi\pare{x}} = O\pare{\pare{-x}^{\frac{1}{2}}}$, and $\frac{1}{2} - ml >0$, we deduce that:
\begin{equation*}
\norme{\pare{\gamma^{1}+i}\varphi\pare{x} \int_{-\infty}^{x} \psi^{t}\pare{y} f\pare{y} dy} = O\pare{\pare{-x}^{\frac{1}{2}}}.
\end{equation*}
For the second term, $\psi$ is in $\mathcal{H}_{m}^{s,n}$ and satisfies:
\begin{equation*}
H_{m}^{s,n} \psi = \lambda \psi.
\end{equation*}
Hence, $\psi \in D_{nat}\pare{H_{m}^{s,n}} = \{ \phi \in \mathcal{H}_{m}^{s,n} \lvert \hspace{2mm} H_{m}^{s,n}\varphi \in \mathcal{H}_{m}^{s,n} \}$. Using the development of $\psi$ near zero obtained in theorem $3.1$ of \cite{IR14}, we can calculate $\pare{\gamma^{1}+i}\psi$ and we obtain $\pare{\gamma^{1}+i}\psi = O\pare{\pare{-x}^{ml}}$. This gives:
\begin{equation*}
\norme{\pare{\gamma^{1}+i}\psi\pare{x}\int_{x}^{0} \varphi^{t}\pare{y} f\pare{y} dy} \leq C \pare{-x}^{\frac{1}{2}}\norme{f}_{L^{2}\pare{]x,0[}}.
\end{equation*}
Thus $\norme{\pare{\gamma^{1}+i}R_{m}^{s,n}\pare{\lambda}f\pare{x}} = O\pare{\pare{-x}^{\frac{1}{2}}}$ and the boundary conditions are satisfied.

\end{proof}
We can now prove the first part of theorem \ref{MainTheo} in the:
\begin{prop}
For all $\lambda \in \C$ such that $\Im\pare{\lambda} >0$, we have:
\begin{equation*}
\pare{H_{m}^{s,n} - \lambda}^{-1} = R_{m}^{s,n}\pare{\lambda}.
\end{equation*}
\end{prop}

\begin{proof}
Recall the relations satisfied by the Dirac matrices:
\begin{equation*}
\gamma^{0^{*}} = \gamma^{0}; \hspace{3mm} \gamma^{j^{*}} = -\gamma^{j},\hspace{3mm} 1\leq j \leq 3;\hspace{3mm} \gamma^{\mu} \gamma^{\nu} + \gamma^{\nu} \gamma^{\mu} = 2 g^{\mu \nu}\mathbf{1},\hspace{3mm} 0\leq \mu, \nu \leq 3.
\end{equation*}
We remark that $\pare{-i} \gamma^{0}\gamma^{1} \gamma^{2}$ commute with $\Gamma^{1}$, $\gamma^{0}\gamma^{2}$ and $\gamma^{0}$ where:
\begin{equation*}
\pare{-i}\gamma^{0} \gamma^{1} \gamma^{2} = \begin{pmatrix}
0 & 0 & 0 & -1 \\
0 & 0 & 1 & 0 \\
0 & 1 & 0 & 0 \\
-1 & 0 & 0 & 0
\end{pmatrix}
\end{equation*}
Consequently, if $\chi$ is such that $H_{m}^{s,n} \chi = \lambda \chi$, then $H_{m}^{s,n} \pare{-i} \gamma^{0} \gamma^{1} \gamma^{2} \chi = \lambda \pare{-i}\gamma^{0} \gamma^{1} \gamma^{2} \chi$. Moreover, $\pare{-i}\gamma^{0} \gamma^{1} \gamma^{2}\chi$ satisfies the boundary conditions if $\chi$ does.\\
Let $\psi = \begin{pmatrix}
\psi_{1} \\
\psi_{2} \\
\psi_{3} \\
\psi_{4}
\end{pmatrix}$ and $\varphi = \begin{pmatrix}
\varphi_{1} \\
\varphi_{2} \\
\varphi_{3} \\
\varphi_{4}
\end{pmatrix}$. Then we have:
\begin{equation*}
\varphi \psi^{t} = \begin{pmatrix}
\varphi_{1} \\
\varphi_{2} \\
\varphi_{3} \\
\varphi_{4}
\end{pmatrix} \begin{pmatrix}
\psi_{1} & \psi_{2} & \psi_{3} & \psi_{4}
\end{pmatrix}  = \begin{pmatrix}
\varphi_{1} \psi_{1} & \varphi_{1} \psi_{2} & \varphi_{1} \psi_{3} & \varphi_{1} \psi_{4} \\
\varphi_{2} \psi_{1} & \varphi_{2} \psi_{2} & \varphi_{2} \psi_{3} & \varphi_{2} \psi_{4} \\
\varphi_{3} \psi_{1} & \varphi_{3} \psi_{2} & \varphi_{3} \psi_{3} & \varphi_{3} \psi_{4} \\
\varphi_{4} \psi_{1} & \varphi_{4} \psi_{2} & \varphi_{4} \psi_{3} & \varphi_{4} \psi_{4}
\end{pmatrix}
\end{equation*}
and similar expressions for $\tilde{\varphi} \tilde{\psi}^{t}$, $\psi \varphi^{t}$ and $\tilde{\psi} \tilde{\varphi}^{t}$.
We obtain:
\begin{equation*}
\varphi \psi^{t} - \psi \varphi^{t} + \tilde{\varphi} \tilde{\psi}^{t} - \tilde{\psi} \tilde{\varphi}^{t} = M_{\alpha,\beta}.
\end{equation*}
We can now express the kernel of our resolvent:
\begin{flalign*}
R_{m}^{s,n}\pare{x,y,\lambda} & = \pare{\varphi\pare{x} \psi^{t}\pare{y} + \tilde{\varphi}\pare{x} \tilde{\psi}^{t}\pare{y}} M_{\alpha,\beta}^{-1} i\Gamma^{1} \mathds{1}_{]-\infty,x[}\pare{y} \\
& \quad + \pare{\psi\pare{x}\varphi^{t}\pare{y} + \tilde{\psi}\pare{x} \tilde{\varphi}^{t}\pare{y}}M_{\alpha,\beta}^{-1} i\Gamma^{1}\mathds{1}_{]x,0[}\pare{y}.
\end{flalign*}
The corresponding operator is given by the following formula:
\begin{equation*}
R_{m}^{s,n}\pare{\lambda}f \pare{x} = \int_{-\infty}^{0} R_{m}^{s,n}\pare{x,y,\lambda}f\pare{y} dy
\end{equation*}
which is well-defined since $\psi$, $\tilde{\psi}$ are exponentially decreasing at $-\infty$ and $\varphi$, $\tilde{\varphi}$ are square integrable near $0$ for any positive value of the mass $m$. Using that $\varphi$, $\psi$, $\tilde{\varphi}$ and $\tilde{\psi}$ satisfy $H_{m}^{s,n} \varphi = \lambda \varphi$, $H_{m}^{s,n}\psi= \lambda \psi$, $H_{m}^{s,n} \tilde{\varphi} = \lambda \tilde{\varphi}$ and $H_{m}^{s,n} \tilde{\psi} = \lambda \tilde{\psi}$, we can calculate:
\begin{equation*}
H_{m}^{s,n} \pare{R_{m}^{s,n}\pare{\lambda}f} \pare{x}  = \lambda R_{m}^{s,n}\pare{\lambda}f\pare{x} + f\pare{x}.
\end{equation*}
for $f\in \left [L^{2}\pare{]-\infty,0[} \right ]^{4}$. Indeed, the first term on the right comes from applying $H_{m}^{s,n}$ to the function $\phi$, $\tilde{\phi}$, $\psi$ and $\tilde{\psi}$ while the second term comes from the differentiation of the integrals. Finally, we have:
\begin{equation*}
\pare{H_{m}^{s,n} - \lambda} \pare{ R_{m}^{s,n}\pare{\lambda} f}\pare{x} = f\pare{x}.
\end{equation*}
Thanks to an adjoint type argument, we also obtain a left inverse which shows the proposition.
\end{proof}

\section{Meromorphic extension of the resolvent}

Let $f_{\epsilon}\pare{x} = e^{\epsilon x}$. In this section, we want to prove the second part of theorem \ref{MainTheo}:
\begin{prop}
The operator $f_{\epsilon} \pare{H_{m}^{s,n} - \lambda}^{-1} f_{\epsilon}$ defined for $\verb?Im? \pare{\lambda} >0$ extends meromorphically to $\{\lambda \in \C \hspace{1mm} \vert \hspace{1mm} \verb?Im? \pare{\lambda} > - \epsilon \}$ for all $0< \epsilon < \frac{\kappa}{2}$ where $\kappa$ is the surface gravity. The poles of this meromorphic extension are called resonances.
\end{prop}

\begin{proof}
Let $0<\epsilon < \frac{\kappa}{2}$. Recall the formula for the resolvent:
\begin{equation*}
R_{m}^{s,n}\pare{\lambda}g \pare{x} = \int_{-\infty}^{0} R_{m}^{s,n}\pare{x,y,\lambda}g\pare{y} dy.
\end{equation*}
for $g \in \mathcal{H}_{s,n}$ with
\begin{flalign*}
R_{m}^{s,n}\pare{x,y,\lambda} & = \pare{\varphi\pare{x} \psi^{t}\pare{y} + \tilde{\varphi}\pare{x} \tilde{\psi}^{t}\pare{y}} M_{\alpha,\beta}^{-1} i\Gamma^{1} \mathds{1}_{]-\infty,x[}\pare{y} \\
& \quad + \pare{\psi\pare{x}\varphi^{t}\pare{y} + \tilde{\psi}\pare{x} \tilde{\varphi}^{t}\pare{y}}M_{\alpha,\beta}^{-1} i\Gamma^{1}\mathds{1}_{]x,0[}\pare{y}
\end{flalign*}
where $M_{\alpha,\beta}$ is given in \eqref{Mabexpr}. Since our operator $H_{m}^{s,n}$ is self-adjoint on his domain, we know that this formula is well defined and analytic for $\Im\pare{\lambda}>0$. We will use this formula to extend $f_{\epsilon} \pare{H_{m}^{s,n} - \lambda}^{-1} f_{\epsilon}$ meromorphically accross the real axis for $\Im\pare{\lambda} > -\epsilon$. Indeed, we can write:
\begin{flalign*}
f_{\epsilon}\pare{x} \pare{\pare{H_{m}^{s,n} - \lambda}^{-1}f_{\epsilon}g} \pare{x} & =f_{\epsilon}\pare{x} \int_{-\infty}^{x} \pare{\varphi\pare{x} \psi^{t}\pare{y} + \tilde{\varphi}\pare{x} \tilde{\psi}^{t}\pare{y}}M_{\alpha,\beta}^{-1} i\Gamma^{1}f_{\epsilon}\pare{y}g\pare{y} dy \\
& \quad + f_{\epsilon}\pare{x}\int_{x}^{0}\pare{\psi\pare{x} \varphi^{t}\pare{y} + \tilde{\psi}\pare{x} \tilde{\varphi}^{t}\pare{y}}M_{\alpha,\beta}^{-1} i\Gamma^{1}f_{\epsilon}\pare{y}g\pare{y} dy.
\end{flalign*}
We will first see that we obtain this way a well-defined operator from $\mathcal{H}_{s,n}$ to $\mathcal{H}_{s,n}$ for $\Im\pare{\lambda} > -\epsilon$.
To this end, we recall that $\psi$ corresponds to the Jost solution while $\varphi$ corresponds to the solution satisfying the boundary condition. The behaviour of $\varphi$ is given by:
\begin{equation*}
\norme{\varphi\pare{x}}\leq 4 N \pare{-x}^{-ml} e^{-\frac{6 C_{\lambda,m} \pare{s+\frac{1}{2}}}{1-2ml} x}
\end{equation*}
for all $x\in ]-\infty,0[$, with $N = 2 \max \pare{{\abs{c},\abs{d}}}$, where $c = \underset{x\to 0}{\lim} \frac{1}{2} \pare{-x}^{ml}\pare{\varphi_{1}-\varphi_{3}}$,\\
$d = \underset{x\to 0}{\lim} \frac{1}{2} \pare{-x}^{ml}\pare{\varphi_{2}+\varphi_{4}}$, and $C_{\lambda,m} = \underset{x\in ]-\infty,0[}{\max} \pare{\abs{\lambda},\abs{m\pare{B\pare{x}+\frac{l}{x}}},\abs{A\pare{x}}}$. When $x \to -\infty$, we have the estimate:
\begin{equation*}
\norme{\varphi \pare{x}} \leq C e^{-\abs{\Im\pare{\lambda}}x}
\end{equation*}
using the same argument as in the proof of proposition \ref{BCS1}. For the Jost solution, using a similar argument as in lemma \ref{BDEDRES}, we see that, near $0$, $\psi$ has the following behaviour:
\begin{equation*}
\norme{\psi\pare{x}} \leq C \pare{-x}^{-ml}
\end{equation*}
which is not singular with respect to $\lambda$. The behaviour of this solution at $-\infty$ is:
\begin{equation*}
\psi \pare{x} = \begin{pmatrix} 0 \\
0 \\
e^{-i\lambda x} \\
0
\end{pmatrix} + o\pare{e^{i\lambda x}}
\end{equation*}
for $0>\Im\pare{\lambda} > -\epsilon$. Since $f_{\epsilon}$ is bounded near $0$, $f_{\epsilon} \varphi$ and $f_{\epsilon} \psi$ behave the same way as $\varphi$ and $\psi$ near $0$. At $-\infty$, $f_{\epsilon} \varphi$ then decay like $e^{\pare{\epsilon - \abs{\Im\pare{\lambda}}}x}$ and while $f_{\epsilon} \psi$ decay like $e^{\pare{\epsilon + \Im\pare{\lambda}}x}$ which ensures the integrability at $-\infty$. This is of course the same for $f_{\epsilon} \tilde{\phi}$ and $f_{\epsilon} \tilde{\psi}$. In the proof of lemma \ref{BDEDRES}, the essential part is the behavior at $-\infty$ and $0$ of $\varphi$ and $\psi$ for $\Im\pare{\lambda}>0$. Since the behavior of $f_{\epsilon} \varphi$ and $f_{\epsilon} \psi$ for $\Im\pare{\lambda} > -\epsilon$ is similar to the behaviour of $\varphi$ and $\psi$ for $\Im\pare{\lambda}>0$, we can use a similar argument as in the proof of lemma \ref{BDEDRES} to obtain the boundedness of $f_{\epsilon} \pare{H_{m}^{s,n} - \lambda}^{-1} f_{\epsilon}$. \\
Now that our operator $f_{\epsilon} \pare{H_{m}^{s,n} - \lambda}^{-1} f_{\epsilon}$ has been extended to the set of $\lambda$ such that $\Im\pare{\lambda} > - \epsilon$, we wish to analyze the analyticity property of this extension. For that, we first remark that, by our construction of $\varphi$ and $\psi$, we know that these functions (and then $\tilde{\varphi}$ and $\tilde{\psi}$) are analytic for $\Im\pare{\lambda} > -\epsilon$. Thus the integral term is analytic. Unfortunately, the matrix $M_{\alpha,\beta}^{-1}$ may have some singularities. These singularities will come from the inverse of the determinant $\det \pare{M_{\alpha,\beta}}  = \pare{\pare{\alpha-\beta}\pare{\alpha +\beta}}^{2}$ which is the inverse of an holomorphic function in $\Im\pare{\lambda} > - \epsilon$. Hence, we have obtained a meromorphic extension of $f_{\epsilon} \pare{H_{m}^{s,n} - \lambda}^{-1} f_{\epsilon}$ for $\Im\pare{\lambda} > -\epsilon$.

\end{proof}

\bibliographystyle{plain-fr}
\bibliography{BiblioResolventFormula}
\nocite{*}

\end{document}